\documentclass[12pt, reqno]{amsart}
\usepackage{amssymb,amsthm,amsfonts,amsmath, amscd}

\usepackage{hyperref}
\usepackage{mathrsfs}

\newcommand\A{\mathbb A}
\newcommand\Y{\mathbb Y}
\newcommand\Z{\mathbb Z}
\newcommand\C{\mathbb C}
\newcommand\R{\mathbb R}

\newcommand\al{\alpha}
\newcommand\be{\beta}
\newcommand\ga{\gamma}
\newcommand\Ga{\Gamma}
\newcommand\de{\delta}

\newcommand\La{\Lambda}
\newcommand\la{\lambda}
\newcommand\si{\sigma}
\newcommand\tth{\theta}
\newcommand\epsi{\varepsilon}
\newcommand\om{\omega}
\newcommand\Om{\Omega}

\newcommand\MM{\mathfrak M}

\newcommand\LL{\mathfrak L}

\newcommand\wt{\widetilde}

\newcommand\const{\operatorname{const}}

\newcommand\Sym{\operatorname{Sym}}
\newcommand\Schur{{\operatorname{Schur}}}

\newcommand\down{\downarrow}
\newcommand\BIN{\operatorname{NBin}}
\newcommand\LA{{\mathsf L}}
\newcommand\ME{{\mathsf M}}
\newcommand\f{{\mathsf f}}
\newcommand\Ch{{\mathsf C}}

\newcommand\D{\mathfrak D}

\newcommand\pd{\partial}
\newcommand\zz{{z,z'}}
\newcommand\zxi{{z,z',\xi}}

\newcommand\FS{F\!S}

\newcommand\ord{{\operatorname{ord}}}

\newtheorem{theorem}{Theorem}[section]
\newtheorem{proposition}[theorem]{Proposition}

\newtheorem{corollary}[theorem]{Corollary}
\newtheorem{lemma}[theorem]{Lemma}

\theoremstyle{definition}
\newtheorem{definition}[theorem]{Definition}
\newtheorem{remark}[theorem]{Remark}

\numberwithin{equation}{section}

\begin{document}

\title[Laguerre and Meixner symmetric functions]
{Laguerre and Meixner  orthogonal bases\\ in the algebra of symmetric
functions}

\author{Grigori Olshanski}
\address{Institute for Information Transmission Problems, Moscow, Russia;
\newline\indent and Independent University of Moscow, Russia} \email{olsh2007@gmail.com}

\date{}

\thanks{Supported by the RFBR-CNRS grant 10-01-93114
and the project SFB 701 of Bielefeld University.}

\begin{abstract} Analogs of Laguerre and Meixner orthogonal polynomials in the
algebra of symmetric functions are studied. The work is motivated by a
connection with a model of infinite-dimensional Markov dynamics.
\end{abstract}

\maketitle

\tableofcontents

\section{Introduction}\label{sect1}

\subsection{Preface}

Let $\Sym$ denote the graded algebra of symmetric functions. The theory of
symmetric functions deals with various homogeneous bases in $\Sym$. A simple
yet fundamental example is the basis of Schur symmetric functions. More
sophisticated examples --- Hall-Littlewood, Jack, and Macdonald symmetric
functions (which are on the top of the hierarchy) --- form a one- or
two-parameter deformation of the Schur functions \cite{Ma95}. Each of these
bases is an orthogonal basis with respect to an appropriate inner product in
$\Sym$.

The aim of the paper is to introduce two new families of orthogonal bases in
$\Sym$: We call them the {\it Laguerre\/} and  {\it Meixner\/} symmetric
functions. In contrast to the symmetric functions mentioned above, the Laguerre
and Meixner symmetric functions are {\it inhomogeneous\/} elements of $\Sym$.
As their names suggest, they are somewhat related to the Laguerre and Meixner
orthogonal polynomials.

Natural multivariate analogs of classical orthogonal polynomials have been
investigated in the widely cited but unpublished manuscript by Macdonald
\cite{Ma87} and in a cycle of notes by Lassalle \cite{La91a}, \cite{La91b},
\cite{La91c}. However, the Laguerre and Meixner symmetric functions cannot be
obtained from the corresponding $N$-variate symmetric polynomials simply by
letting $N\to\infty$, like as the Schur symmetric functions arise from the
$N$-variate symmetric Schur polynomials.

The construction of the present paper is based on the following trick: we treat
$N$, the number of variables, as an independent parameter and then perform
analytic continuation into complex domain with respect to this parameter. As a
result, the Laguerre symmetric functions depend on two parameters while the
Laguerre polynomials involve a single parameter only, and the Meixner symmetric
functions acquire three parameters instead of the conventional two parameters.

Another feature of the construction is that the most natural realization of the
Laguerre and Meixner symmetric functions is achieved when $\Sym$ is realized as
the algebra of {\it supersymmetric functions\/}.

\medskip

\subsection{Description of results}

1) As usual in $\Sym$, the basis elements are indexed by arbitrary partitions
$\nu$. The Laguerre symmetric functions $\LL_\nu$ are defined in terms of their
expansion in the basis $\{S_\nu\}$ of  Schur functions, see \eqref{eq5}. A
similar expansion for the Meixner symmetric functions $\MM_\nu$ is \eqref{eq6},
but here the Schur functions are replaced by their factorial analogs, the
so-called Frobenius-Schur functions $\FS_\nu$. These expansions may be viewed
as analogs of the well-known explicit formulas for the univariate Laguerre and
Meixner polynomials.

\medskip

2) Alternatively, the elements $\LL_\nu$  and $\MM_\nu$ can be characterized by
two properties:

\begin{itemize}

\item First,
$$
\LL_\nu=S_\nu\,+\,\textrm{lower degree terms}, \quad
\MM_\nu=S_\nu\,+\,\textrm{lower degree terms}.
$$

\item Second, $\LL_\nu$  and $\MM_\nu$ are eigenfunctions of some operators in
$\Sym$, the Laguerre operator $\D^\LA$ and the Meixner operator $\D^\ME$,
respectively:
$$
\D^\LA\LL_\nu=-|\nu|\LL_\nu, \quad \D^\ME\MM_\nu=-|\nu|\MM_\nu.
$$

\end{itemize}

\noindent This characterization is similar to the well-known characterization
of classical orthogonal polynomials as polynomial eigenfunctions of suitable
differential or difference operators of hypegeometric type.

\medskip

3) Operators $\D^\LA$ and $\D^\ME$ are initially defined by their action on the
Schur and Frobenius-Schur functions, see \eqref{LaMe.4} and \eqref{eq8}.
Alternatively, the Laguerre operator $\D^\LA$ can be written as a second order
differential operator in formal variables $e_1, e_2,\dots$ (the elementary
symmetric functions) or $h_1,h_2,\dots$ (the complete homogeneous symmetric
functions), see Theorem \ref{thm1} and Corollary \ref{cor4}. As for Meixner
operator $\D^\ME$, it can be written as a kind of difference operator on the
Young graph $\Y$, see Proposition \ref{prop10}.

\medskip

4) We define inner products in $\Sym$ in which the functions $\LL_\nu$ and
$\MM_\nu$ are orthogonal. This is done in terms of formal moment functionals
$\varphi^\LA$ and $\varphi^\ME$ on the space $\Sym$, which are explicitly
computed in Proposition \ref{prop19}. We also get explicit expressions for the
squared norms of $\LL_\nu$ and $\MM_\nu$, see Theorem \ref{thm2}.

\medskip

5) To achieve a full analogy with orthogonal polynomials we have to show that
the formal moment functionals can be written as the expectations under some
probability  measures, in an appropriate functional realization of the algebra
$\Sym$. Such measures (we call them the {\it orthogonality measures\/}) are
exhibited in Theorem \ref{thm4} and Theorem \ref{thm3}.

\medskip

6) Finally, as is well known, the Meixner polynomials are discrete counterparts
of the Laguerre polynomials: the latter are limits of the former. Likewise, the
Laguerre differential operator is a scaling limit of the Meixner difference
operator, and the continuous gamma distribution (the weight measure for the
Laguerre polynomials)  can be obtained through a scaling limit from the
discrete negative binomial distribution (the weight for the Meixner
polynomials). We show that similar limit relations hold for the Laguerre and
Meixner symmetric functions. In particular, Theorem \ref{thm5} explains how the
Laguerre orthogonality measure can be approximated by the Meixner orthogonality
measures.

\medskip

7) In Appendix, we briefly describe a degeneration of the Meixner symmetric
functions leading to one more family of symmetric functions, which we call the
Charlier functions. The corresponding orthogonal measure is the well-known
poissonized Plancherel measure.

\smallskip

Note that the Laguerre and Meixner orthogonality measures have a
representation-theoretic origin. They were earlier studied in a cycle of papers
by Borodin and the author, see \cite{BO00a}, \cite{BO00b}, \cite{BO06a},
\cite{BO06c}, \cite{BO09}, \cite{Ol03a}, \cite{Ol03b}. The Meixner
orthogonality measures are the so-called {\it z-measures\/} on $\Y$. The
Laguerre orthogonality measures arise in noncommutative harmonic analysis on
the infinite symmetric group: they govern the spectral decomposition of some
unitary representations. The Laguerre orthogonality measures live on an
infinite-dimensional cone $\wt\Om$ whose base is Thoma's simplex $\Om$ (a kind
of dual space to the infinite symmetric group).

Note also that both the Laguerre and Meixner orthogonality measures give rise
to determinantal point processes, see \cite{BO00a}, \cite{BO06a}, \cite{BO06c}.

Next, as is shown in \cite{BO06a}, the Meixner difference operator on the Young
graph $\Y$ serves as the infinitesimal generator of a jump Markov process. A
similar (but more difficult result) holds in the Laguerre case, too: the
Laguerre operator generates a diffusion process on the cone $\wt\Om$, see
\cite{Ol10b}. The fact that the Laguerre and Meixner symmetric functions
provide an explicit diagonalization of these Markov generators was the main
motivation of the present work.

Finally, the Charlier operator \eqref{eq25}, which is a degeneration of the
Meixner operator, serves as the infinitesimal generator in a model of Markov
dynamics related to the Plancherel measure, see \cite{BO06b}.

\subsection{Notes}\label{sect1.A}

The present paper is a detailed exposition of part of the results announced in
\cite{Ol10b}. Other results of \cite{Ol10b}, which concern the ``Laguerre''
Markov dynamics on the Thoma cone, are the subject of a separate detailed paper
under preparation.

In the study of the z-measures, the trick of analytic continuation in $N$ was
exploited in \cite{BO06a} and \cite{BO06c}; in that papers we developed some
observations made earlier in \cite[\S4 and Remark 5.5]{BO00a}.

In the context of symmetric functions, analytic continuation in $N$ was earlier
used by Rains \cite{Ra05} and then by Sergeev and Veselov \cite{SV09a},
\cite{SV09b}, \cite{SV09c} to lift $BC_N$-type symmetric polynomials to
infinite dimensions. Rains' formidable paper deals with the Koornwinder
polynomials (a multivariate version of the Askey-Wilson polynomials), and
Sergeev and Veselov deal with the Jacobi polynomials.

The approach of Sergeev and Veselov was extended to the case of Hermite and
Laguerre symmetric functions by Desrosiers and Halln\"as \cite{DH11}. It is
interesting to compare their results and those of the present paper. There are
some intersections in what concerns the Laguerre case; the construction of
\cite{DH11} is more general as it involves the extra (Jack) parameter
$\alpha=1/\theta$; \cite{DH11} also contains a number of nice additional
results about the Laguerre symmetric functions. On the other hand, Desrosiers
and Halln\"as do not discuss the Meixner case and the topics related to
orthogonality.

Note that one can extend the construction of the Meixner symmetric functions to
the case of the Jack parameter $\theta=1/\alpha$. However, in my understanding,
these more general Meixner functions no longer live in $\Sym$, but should be
defined as elements of a certain algebra of functions on $\Y$ (the algebra
$\A_\theta$ of $\theta$-regular functions on $\Y$, in the terminology of
\cite{Ol10a}; it is isomorphic to the algebra of $\theta$-shifted symmetric
functions from \cite{OO97b}). In the special case $\theta=1$, there is a
natural isomorphism between this algebra and algebra $\Sym$ (see \cite{ORV03})
but it seems that there is no distinguished way to identify $\A_\theta$ with
$\Sym$ when $\theta\ne1$.

The importance of the supersymmetric realization of the algebra $\Sym$ for the
representation theory of the infinite symmetric group became clear after the
work of Vershik and Kerov \cite{VK81}, \cite{VK90}. For subsequent developments
of their ideas, see \cite{KOO98}, \cite{ORV03}, \cite{IO03}; see also numerous
papers by Sergeev and Veselov on integrable systems related to Lie
superalgebras.

\section{Classical Laguerre and Meixner polynomials}\label{class}

Here we collect a few necessary formulas concerning the classical univariate
Laguerre and Meixner polynomials. All these formulas are easily extracted from
\cite{KS96}. The reader should keep in mind that our normalization of these
polynomials and our notation differ from the conventional ones, as we prefer to
work with {\it monic\/} polynomials.

The {\it classical Laguerre polynomials\/} $L_n(x)$ depend on a parameter
$b>0$; they form an orthogonal system on the half-line $\R_+:=\{x\in\R\colon
x\ge0\}$ with respect to the weight measure
\begin{equation}\label{eq29}
\ga_b(dx)=\frac1{\Ga(b)}x^{b-1}e^{-x}dx, \quad x\in\R_+,
\end{equation}
which is the probability {\it gamma distribution\/}.

The second order differential operator
\begin{equation*}
D^\LA =x\frac{d^2}{d x^2}+(b-x)\frac{d}{dx},
\end{equation*}
is formally symmetric with respect to $\ga_b$, and the Laguerre polynomials are
eigenfunctions of $D^\LA$:
\begin{equation*}
D^\LA L_n=-n L_n.
\end{equation*}
Moreover, they are the only polynomial eigenfunctions of $D^\LA$. Note a useful
formula for the action of $D^\LA$ on the monomials:
\begin{equation}\label{class.A}
D^\LA x^n=-n x^n+n(n+b-1)x^{n-1}.
\end{equation}

Introduce a notation for the falling factorial powers of a variable:
\begin{equation*}
x^{\down m}=x(x-1)\dots (x-m+1)=(-1)^m(-x)_m, \quad m=0,1,\dots
\end{equation*}
with the understanding that $x^{\down 0}=1$. Here $(x)_m=x(x+1)\dots(x+m-1)$ is
the standard notation for the raising factorial power, aka the {\it Pochhammer
symbol\/}. In this notation, the monic Laguerre polynomials can be explicitly
written as
\begin{equation}\label{class.F}
L_n(x)=(-1)^n (b)_n\,{}_1F_1(-n;b;x) =(b)_n\sum_{m=0}^n(-1)^{m-n}\frac{n^{\down
m}}{(b)_m m!}x^m,
\end{equation}
where ${}_1F_1$ is the confluent hypergeometric series.

An important constant is the squared norm with respect to the weight measure,
\begin{equation}\label{eq1}
(L_n,L_n):=\int_0^{+\infty}(L_n(x))^2\ga_b(dx)=(b)_n n!.
\end{equation}

The {\it classical Meixner polynomials\/} $M_n(x)$ depend on the same parameter
$b>0$ and an additional parameter $\xi\in(0,1)$. They are orthogonal with
respect to the {\it negative binomial distribution\/} on the half-lattice
$\Z_+:=\{x\in\Z\colon x\ge0\}$ given by
\begin{equation*}
\eta_{b,\xi}=(1-\xi)^b\sum_{x\in\Z_+}\frac{(b)_x}{x!}\xi^x\de_x,
\end{equation*}
where $\de_x$ denotes the Dirac measure at $x$.

One associates with the family $\{M_n\}$ the second order difference operator
on $\Z_+$ defined by
\begin{equation*}
\begin{aligned}
D^\ME f(x)& =\frac{\xi(b+x)}{1-\xi}f(x+1)+\frac x{1-\xi}f(x-1)\\
&-\frac{\xi(b+x)+x}{1-\xi}f(x).
\end{aligned}
\end{equation*}
It is formally symmetric with respect to the weight function $\eta_{b,\xi}$ and
annihilates the constants. The Meixner polynomials can be characterized as the
only polynomial eigenfunctions of this operator:
\begin{equation*}
D^\ME M_n=-n M_n.
\end{equation*}

The operator $D^\ME$ admits a nice expression in the inhomogeneous basis
$\{x^{\down n}\}$:
\begin{equation}\label{class.B}
D^\ME x^{\down n}=-n x^n+\frac{\xi}{1-\xi}n(n+b-1)x^{\down(n-1)}.
\end{equation}

Here is an explicit expression of the monic Meixner polynomials:
\begin{equation}\label{class.E}
\begin{aligned}
M_n(x)&=\left(\frac\xi{\xi-1}\right)^n(b)_n\,{}_2F_1\left(-n,
-x;b;\frac{\xi-1}\xi\right)\\
&=(b)_n\sum_{m=0}^n\left(\frac\xi{\xi-1}\right)^{n-m}\frac{n^{\down m}} {(b)_m
m!}x^{\down m},
\end{aligned}
\end{equation}
where ${}_2F_1$ is the Gauss hypergeometric series.

The squared norm with respect to the weight $\eta_{b,\xi}$ on $\Z_+$ is given
by
\begin{equation}\label{eq2}
(M_n,M_n)=\xi^n(1-\xi)^{-2n}(b)_n n!.
\end{equation}

\begin{proposition}\label{class.1}
In a scaling limit transition, as $\xi\to1$ and simultaneously the lattice is
shrunk with small coefficient $1-\xi$, the Meixner polynomials converge to the
Laguerre polynomials:
\begin{equation*}
\lim_{\xi\to1}(1-\xi)^n M_n((1-\xi)^{-1}x)=L_n(x).
\end{equation*}
\end{proposition}

Note that the multiplication by $(1-\xi)^n$ is needed to keep the top degree
coefficient of the polynomial to be equal to $1$.

\begin{proof}
Compare \eqref{class.E} and \eqref{class.F}. Set $\epsi=1-\xi$ and
$$
x^{\down m,\epsi}=x(x-\epsi)\dots(x-(m-1)\epsi).
$$
Substituting $x\to(1-\xi)^{-1}x$ in \eqref{class.E} and multiplying by
$(1-\xi)^n$ results in replacing $x^{\down m}$ by $x^{\down m,\epsi}$. It
follows that the desired convergence holds term-wise, for every fixed
$m=0,\dots,n$.
\end{proof}

\begin{remark}\label{rem1}
(i) The above asymptotic relation is explained by the fact that in the same
scaling limit regime, the negative binomial distribution approximates the gamma
distribution.

(ii) An alternative explanation can be extracted from the comparison of
\eqref{class.A} and \eqref{class.B}. As above, make the change of a variable
$x\to\epsi^{-1}x$, which means that the lattice $\Z_+$ is shrunk with small
scale factor $\epsi$. The analog of \eqref{class.B} for the rescaled difference
operator $D^{\ME,\epsi}$ is
\begin{equation*}
D^{\ME,\epsi} x^{\down n,\epsi}=-n x^{\down n,\epsi}+\xi n(n+b-1)x^{\down
(n-1),\epsi}.
\end{equation*}
In the limit as $\epsi\to0$ this leads to \eqref{class.A}.
\end{remark}

\section{The $N$-variate Laguerre and Meixner symmetric polynomials}

\subsection{{}From univariate polynomials to $N$-variate symmetric polynomials}

Fix $N=1,2,\dots$ and denote by $\Sym(N)$ the subalgebra of symmetric
polynomials in $\R[x_1,\dots,x_N]$. Let $\Y(N)$ denote the set of all integer
partitions with at most $N$ nonzero summands; we write such partitions as
$N$-dimensional vectors
\begin{equation*}
\nu=(\nu_1,\dots,\nu_N)\in\Z_+^N, \qquad \nu_1\ge\dots\ge\nu_N,
\end{equation*}
and identify them with Young diagrams with at most $N$ nonzero rows. The
quantity $|\nu|:=\nu_1+\dots+\nu_N$ is equal to the number of boxes in $\nu$.
By $\varnothing$ we denote the empty Young diagram represented by the zero
vector $(0,\dots,0)$, and $V_N$ is our shorthand notation for the Vandermonde:
\begin{equation*}
V_N=V_N(x_1,\dots,x_N):=\prod_{1\le i<j\le N}(x_i-x_j).
\end{equation*}

Let $\{\phi_n\}_{n=0,1,\dots}$ be an arbitrary basis in $\R[x]$ formed by monic
polynomials with $\deg \phi_n=n$. There is a well-known way to construct from
$\{\phi_n\}$ a basis $\{\phi_{\nu\mid N}\}$  in $\Sym(N)$ labelled by Young
diagrams $\nu\in\Y(N)$. Namely, for any $\nu\in\Y(N)$,
\begin{equation}\label{Ndim.G}
\phi_{\nu\mid
N}(x_1,\dots,x_N):=\frac{\det[\phi_{\nu_i+N-i}(x_j)]}{V_N(x_1,\dots,x_N)}
\end{equation}
is a symmetric polynomial, because the $N\times N$ determinant in the numerator
is an antisymmetric polynomial and hence is divisible by $V_N$. As $\nu$ ranges
over $\Y(N)$, we get the desired  basis in $\Sym(N)$. Note that
\begin{equation*}
\deg \phi_{\nu\mid N}=|\nu|.
\end{equation*}

We will employ this construction in four particular cases:

\medskip

$\bullet$ Setting $\phi_n=x^n$ produces the  {\it Schur polynomials\/}, which
we denote by $S_{\nu\mid N}$.

\medskip

$\bullet$ Setting
$$
\phi_n=x^{\down n}:=x(x-1)\dots(x-n+1)
$$
produces the so-called {\it factorial Schur polynomials\/} denoted as
$S^\f_{\nu\mid N}$.

\medskip

$\bullet$ Setting $\phi_n=L_n$ produces the {\it $N$-variate Laguerre
polynomials\/} denoted as $L_{\nu\mid N}$ or, in more detail, $L_{\nu\mid
N,b}$.

\medskip

$\bullet$ Setting $\phi_n=M_n$ produces the {\it $N$-variate Meixner
polynomials\/} denoted as $M_{\nu\mid N}=M_{\nu\mid N, b,\xi}$.

\medskip

The Schur polynomials are homogeneous, the other polynomials $\phi_{\nu\mid N}$
are not, but the top degree homogeneous component in $\phi_{\nu\mid N}$ always
coincides with $S_{\nu\mid N}$:
\begin{equation*}
\phi_{\nu\mid N}= S_{\nu\mid N}+\textrm{lower degree terms}.
\end{equation*}
It follows  that every family of the form $\{\phi_{\nu\mid N}\}$ is a basis in
$\Sym(N)$. Moreover, any family $\{\phi_{\nu\mid N}\}$ is consistent with the
canonical filtration in $\Sym(N)$ in the sense that, for any $n=0,1,2,\dots$,
the set of the basis elements of degree $\le n$ form a basis in the space of
symmetric polynomials of degree $\le n$.

\subsection{Expansions}

Assume $\{\phi'_n\}$ is another basis in $\R[x]$ formed by monic polynomials,
and let $t(n,m)$ denote the transition coefficients between the two bases,
\begin{equation*}
\phi_n=\sum_{m=0}^n t(n,m)\phi'_m.
\end{equation*}

\begin{proposition}\label{prop4}
Let\/ $\{\phi_n\}$, $\{\phi'_n\}$, and\/ $t(n,m)$ be as above. For any
$\nu\in\Y(N)$ we have the expansion
\begin{equation*}
\phi_{\nu\mid N}=\sum_{\mu:\,\mu\subseteq\nu}t(\nu,\mu)\phi'_{\mu\mid N},
\end{equation*}
where summation is taken over all diagrams $\mu$ contained in $\nu$ and the
coefficients are $N\times N$ determinants built from the numbers $t(n,m)$,
\begin{equation*}
t(\nu,\mu):=\det[t(n_i,m_j)], \qquad n_i=\nu_i+N-i, \quad m_j=\mu_j+N-j.
\end{equation*}
\end{proposition}

\begin{proof}
Direct verification.
\end{proof}

For $\nu\in\Y(N)$ and $\mu\subseteq\nu$, denote by $\dim\nu/\mu$ the number of
standard tableaux of skew shape $\nu/\mu$.
\begin{corollary}
The following expansions hold
\begin{equation}\label{Ndim.C}
L_{\nu\mid N,b}=\sum_{\mu:\,\mu\subseteq\nu} (-1)^{|\nu|-|\mu|}
\frac{\dim\nu/\mu}{(|\nu|-|\mu|)!}
\,\prod_{i=1}^N\frac{(\nu_i+N-i)!(b)_{\nu_i+N-i}}{(\mu_i+N-i)!(b)_{\mu_i+N-i}}\cdot
S_{\mu\mid N}
\end{equation}
and
\begin{multline}\label{Ndim.D}
M_{\nu\mid N,b,\xi}=\sum_{\mu:\,\mu\subseteq\nu}
(-1)^{|\nu|-|\mu|}\left(\frac\xi{1-\xi}\right)^{|\nu|-|\mu|}
\frac{\dim\nu/\mu}{(|\nu|-|\mu|)!}\\
\times
\prod_{i=1}^N\frac{(\nu_i+N-i)!(b)_{\nu_i+N-i}}{(\mu_i+N-i)!(b)_{\mu_i+N-i}}\cdot
S^\f_{\mu\mid N}
\end{multline}
\end{corollary}

\begin{proof}
Apply Proposition \ref{prop4} for  $\phi_n=L_n$ and $\phi'_n=x^n$. By
\eqref{class.F}
\begin{equation*}
t(n,m)=(-1)^{n-m}\frac{(b)_n n^{\down m}}{(b)_m m!}.
\end{equation*}
Then we get, setting $n_i=\nu_i+N-i$ and $m_j=\mu_j+N-j$,
\begin{equation*}
\begin{aligned}
t(\nu,\mu)&=(-1)^{|\nu|-|\mu|}\prod_{i=1}^n\frac{(b)_{n_i}}{(b)_{m_i}m_i!}\cdot
\det[n_i^{\down m_j}]\\
&=(-1)^{|\nu|-|\mu|}\prod_{i=1}^n\frac{n_i!(b)_{n_i}}{m_i!(b)_{m_i}}\cdot
\frac{\det[n_i^{\down m_j}]}{\prod n_i!}
\end{aligned}
\end{equation*}
On the other hand,
\begin{equation*}
\frac{\det[n_i^{\down m_j}]}{\prod
n_i!}=\det\left[\frac1{(n_i-m_j)!}\right]=\frac{\dim\nu/\mu}{(|\nu|-|\mu|)!},
\end{equation*}
see, e.g., the proof of Proposition 1.2 in \cite{ORV03} (in that proof there is
a minor typo: the correct reference to Macdonald's book should be \cite[\S I.7,
Ex. 6]{Ma95}). This gives \eqref{Ndim.C}.

The proof of \eqref{Ndim.D} is just the same: we take $\phi_n=M_n$ and
$\phi'_n=x^{\down n}$, and use \eqref{class.B}.
\end{proof}

\subsection{Differential/difference operators}

Let $\{\phi_n\}$ be as above. Consider the linear operator
$D\colon\R[x]\to\R[x]$ defined by $D \phi_n=-n \phi_n$ for $n=0,1,2,\dots$.

\begin{proposition}
Fix $N=1,2,\dots$ and denote by $D^{(i)}$ a copy of $D$ acting on the $i$th
variable $x_i$, $1\le i\le N$. The formula
\begin{equation}\label{Ndim.H}
D_N=V_N^{-1}\circ\left(D^{(1)}+\dots+D^{(N)}\right)\circ V_N +\tfrac{N(N-1)}2,
\end{equation}
correctly determines a linear operator\/ $\Sym(N)\to\Sym(N)$, and we have
\begin{equation}\label{Ndim.A}
D_N \phi_{\nu\mid N}=-|\nu|\phi_{\nu\mid N}, \qquad \forall\nu\in\Y(N).
\end{equation}
In particular, $D_N$ annihilates the constants.
\end{proposition}

Note that in the simplest case $N=1$, operator $D_N$ reduces to $D$.

\begin{proof}
Multiplication by $V_N$ takes a symmetric polynomial to an antisymmetric one.
Application of the symmetric operator $\sum D_{x_i}$ produces another
antisymmetric polynomial. It can be divided by $V_N$, which finally results in
another symmetric polynomial. Therefore, $D_N$ is correctly defined in
$\Sym(N)$.

It follows from \eqref{Ndim.G} that the operator $V_N^{-1}\circ\left(\sum
D_{x_i}\right)\circ V_N$ multiplies $\phi_\nu$ by
$$
-\sum_{i=1}^N(\nu_i+N-i)=-|\nu|-\tfrac{N(N-1)}2,
$$
and then the  second term is cancelled by the constant term $N(N-1)/2$ in
$D_N$, leading finally to $-|\nu|$, as stated.

In particular, $D_N\phi_\varnothing=0$, which means that $D_N1=0$.
\end{proof}

For $D=D^\LA$ or $D=D^\ME$, we will denote the corresponding operators
$D_N:\Sym(N)\to\Sym(N)$ by $D^\LA_N$ or $D^\ME_N$, respectively. The relation
\eqref{Ndim.A} now turns into
\begin{equation}\label{Ndim.F}
D^\LA_N L_{\nu\mid N,b}=-|\nu|L_{\nu\mid N,b}, \quad D^\ME_N M_{\nu\mid
N,b,\xi}=-|\nu|M_{\nu\mid N,b,\xi}, \qquad \nu\in\Y(N).
\end{equation}

In the next proposition we interpret $D^\ME_N$ as a partial difference operator
living on the discrete set formed by ordered $N$-tuples of points in $\Z_+^N$,
\begin{equation}\label{Ndim.I}
\Z^N_{+,\,\ord}:=\{x=(x_1,\dots,x_N)\in\Z_+^N\colon x_1>\dots>x_N\}.
\end{equation}
Note that any polynomial in variables $x_1,\dots,x_N$ is uniquely determined by
its restriction to $\Z^N_{+,\,\ord}$.

\begin{proposition}\label{prop9}
Upon restriction to\/ $\Z^N_{+,\,\ord}$, the operator
$D^\ME_N:\Sym(N)\to\Sym(N)$ is implemented by the difference operator
\begin{equation}\label{eq27}
\begin{aligned}
D^\ME_Nf(x)&=\sum_{i=1}^N A_i(x)f(x+\epsi_i)+\sum_{i=1}^N B_i(x)f(x-\epsi_i) -
C(x)f(x)\\
&=\sum_{i=1}^N A_i(x)(f(x+\epsi_i)-f(x))+\sum_{i=1}^N
B_i(x)(f(x-\epsi_i)-f(x)).
\end{aligned}
\end{equation}
Here $f(x)$ is an arbitrary function on $\Z^N_{+,\,\ord}$,
$\{\epsi_1,\dots,\epsi_N\}$ is the canonical basis in $\R^N$, and the
coefficients are given by
\begin{equation}\label{Ndim.K}
\begin{aligned}
A_i(x)&=\frac{\xi(b+x_i)}{1-\xi}\frac{V_N(x+\epsi_i)}{V_N(x)},\\
B_i(x)&=\frac{x_i}{1-\xi}\frac{V_N(x-\epsi_i)}{V_N(x)}\\
C(x)&=\frac{\xi b N+(1+\xi)\sum_{i=1}^Nx_i}{1-\xi}-\frac{N(N-1)}2,
\end{aligned}
\end{equation}
where $V_N(x)=V_N(x_1,\dots,x_N)$.
\end{proposition}

Note that  if $x+\epsi_i$ or $x-\epsi_i$ falls outside $\Z^N_{+,\,\ord}(N)$,
then the corresponding coefficient automatically vanishes.

\begin{proof}
Easy direct check.
\end{proof}

Likewise, the next proposition says that $D^\LA_N$ can be explicitly written as
a partial differential operator.

\begin{proposition}
Upon restriction to the open cone
$$
\R^N_{>0,\,\ord}:=\{(x_1,\dots,x_n)\in\R^N: x_1>\dots>x_N>0\}
$$
the operator  $D^\ME_N:\Sym(N)\to\Sym(N)$ is implemented by the partial
differential operator
\begin{equation*}
D^\LA_N=\sum_{i=1}^Nx_i\frac{\pd^2}{\pd
x_i^2}+\sum_{i=1}^N\left(b-x_i+\sum_{j:\,j\ne
i}\frac{2x_i}{x_i-x_j}\right)\frac{\pd}{\pd x_i},
\end{equation*}
\end{proposition}

\begin{proof}
Direct computation.
\end{proof}

Finally, write down the action of $D^\LA_N$ and $D^\ME_N$ on the Schur
polynomials and on the factorial Schur polynomials, respectively:

For $\nu\in\Y(N)$ and $i=1,\dots,N$, denote by $\nu-\epsi_i$ the vector
obtained from $\nu=(\nu_1,\dots,\nu_N)$ by decreasing the $i$th coordinate by
$1$. If $\nu-\epsi_i$ is not a partition, then we agree that
$S_{\nu-\epsi_i\mid N}=S^\f_{\nu-\epsi_i\mid N}=0$. As above, we set
$n_i=\nu_i+N-i$.

\begin{proposition}
In this notation, we have
\begin{gather}
D^\LA_N S_{\nu\mid N}=-|\nu|S_{\nu\mid N} +\sum_{i=1}^N n_i(n_i+b-1)S_{\nu-\epsi_i\mid N}
\label{Ndim.B}\\
D^\ME_N S^\f_{\nu\mid N}=-|\nu|S^\f_{\nu\mid N}+\frac{\xi}{1-\xi}\sum_{i=1}^N
n_i(n_i+b-1)S^\f_{\nu-\epsi_i\mid N}.\label{Ndim.E}
\end{gather}
\end{proposition}

\begin{proof}
Let us check the first relation. By \eqref{Ndim.H}, $D^\LA_N S_{\nu\mid N}$ is
the sum of two terms: one is $\frac {N(N-1)}2 S_{\nu\mid N}$ and the other is
$$
\left\{V_N^{-1}\circ\left(D^{(1)}+\dots+D^{(N)}\right)\circ V_N\right\}
S_{\nu\mid N},
$$
where $D^{(1)}, \dots,D^{(N)}$ are copies of the one-dimensional Laguerre
operator $D^\LA$ acting on variables $x_1,\dots,x_N$, respectively. By the
definition of $S_{\nu\mid N}$, the last expression equals
$$
\dfrac{\left(D^{(1)}+\dots+D^{(N)}\right)\det\left[x_j^{n_k}\right]_{j,k=1}^N}{V_N(x)}.
$$
Expanding the determinant and applying formula \eqref{class.A} for the action
of $D^\LA$ on monomials in one variable, we get after simple transformations
$$
-(n_1+\dots+n_N)\dfrac{\det\left[x_j^{n_k}\right]_{j,k=1}^N}{V_N(x)}
+\sum_{i=1}^Nn_i(n_i+b-1)\dfrac{\det\left[x_j^{n_k-\de_{ik}}\right]_{j,k=1}^N}{V_N(x)}.
$$
Since
$$
-(n_1+\dots+n_N)+\frac{N(N-1)}2=-|\nu|,
$$
this gives the desired result.

The second relation is checked in the same way.
\end{proof}

\subsection{Orthogonality}

Let again $\{\phi_n\}$ be a sequence of real monic polynomials with
$\deg\phi_n=n$. Assume that there exists a probability measure $w$ on $\R$ with
finite moments of all orders and such that the polynomials $\phi_n$ form an
orthogonal system in the weighted Hilbert space $L^2(\R,w)$:
\begin{equation*}
(\phi_m,\phi_n):=\int \phi_m(x)\phi_n(x) w(dx)=\de_{mn}\cdot \const_n, \qquad
m,n\in\Z_+.
\end{equation*}

Given an arbitrary $N=1,2,\dots$, we use the shorthand notation
$x:=(x_1,\dots,x_N)$, $V_N(x):=V_N(x_1,\dots,x_N)$,  and set
\begin{equation*}
\R^N_\ord:=\{x=(x_1,\dots,x_N)\in\R\colon x_1>\dots>x_N\}.
\end{equation*}
Points of $\R^N_\ord$ can be interpreted as {\it $N$-particle configurations\/}
in $\R$.

Next, introduce an inner product in $\Sym(N)$ by setting for $F,G\in\Sym(N)$
\begin{align}
(F,G)_N&=\frac1{N!}\prod_{i=1}^N\frac1{(\phi_{N-i},\phi_{N-i})}\int\limits_{\R^N}
F(x)G(x)V_N^2(x)\prod w(dx_i)\label{orth.A}\\
&=\prod_{i=1}^N\frac1{(\phi_{N-i},\phi_{N-i})}\int\limits_{R^N_\ord}
F(x)G(x)V_N^2(x)\prod w(dx_i)\label{orth.B}
\end{align}
The expressions in \eqref{orth.A} and \eqref{orth.B} are the same for the
following reasons: First, since $V_N(x)$ vanishes on the diagonal hyperplanes
$x_i=x_j$, we may restrict integration in \eqref{orth.A} to the complement to
all such hyperplanes, even if the measure $w$ has atoms. Second, since $F$ and
$G$ are symmetric, we may further restrict integration to $\R^N_\ord$ by
introducing the extra factor $N!$, which will cancel the same factorial in the
denominator.

\begin{proposition}
The $N$-variate symmetric polynomials $\phi_{\nu\mid N}$ are orthogonal with
respect to the above inner product. More precisely,
\begin{equation}
(\phi_{\mu\mid N},\phi_{\nu\mid N})_N=\de_{\mu\nu}\cdot\prod_{i=1}^N
\frac{(\phi_{n_i},\phi_{n_i})}{(\phi_{N-i},\phi_{N-i})}
\end{equation}
where $\mu,\nu\in\Y(N)$ and $n_i=\nu_i+N-i$.
\end{proposition}

Note that for the constant function $\phi_{\varnothing\mid N}\equiv1$, one has
$n_i=N-i$, so that the squared norm equals $(\phi_{\varnothing\mid
N},\phi_{\varnothing\mid N})_N=1$.

\begin{proof}
Direct computation: Use the definition of the polynomials as ratios of
determinants and observe that the $V^2_N$ factor in the integrand cancels the
denominators. Then expand the determinants in the numerators and integrate out
taking into account orthogonality of the univariate polynomials.
\end{proof}

Obviously, the inner product $(F,G)_N$ in $\Sym(N)$ defined above coincides
with the integral of $FG$ against the following measure on $\R^N_\ord$
\begin{equation}\label{eq14}
w_N:=\prod_{i=1}^N\frac1{(\phi_{N-i},\phi_{N-i})}\cdot V^2_N\prod
w(dx_i)\bigg|_{\R^N_\ord}.
\end{equation}
Since $(1,1)_N=1$,  $w_N$ is a probability measure on $\R^N_\ord$. The
probability space $(\R^N_\ord,w_N)$ gives rise to random $N$-particle
configurations that constitute the {\it $N$-particle orthogonal polynomial
ensemble\/} corresponding to the system $\{\phi_n\}$ of orthogonal polynomials
with weight $w$, see \cite{Ko05}.

\begin{proposition}
Assume $w$ is such that the space of polynomials is dense in $L^2(\R,w)$. Then
the $N$-variate polynomials $\phi_\nu$, $\nu\in\Y(N)$, form an orthogonal basis
in $L^2(\R^N_\ord,w_N)$ for every $N$.
\end{proposition}

\begin{proof}
The operator of multiplication by $V_N$ maps $L^2(\R^N_\ord,w_N)$ isometrically
onto the Hilbert space
$$
L^2(\R^N_\ord, \const\prod w(dx_i))
$$
with an appropriate constant factor in front of the product measure $\prod
w(dx_i)$. Further, the latter space can be identified with the subspace of
antisymmetric functions in the Hilbert space $L^2(\R^N,\const\prod w(dx_i))$,
where antisymmetry is understood with respect to the action of the symmetric
group $S_n$ by permutations of coordinates. By the assumption on $w$, any
function in the latter Hilbert space can be approximated, in the Hilbert norm,
by $N$-variate polynomials. Applying antisymmetrization, we get that any
antisymmetric function can be approximated by antisymmetric polynomials.
Finally, any such polynomial is the product of a symmetric polynomial by $V_N$.
This concludes the proof.
\end{proof}

Now we specialize the proposition to the case of Laguerre and Meixner
polynomials; then we take as $w$ the gamma distribution on $\R_+\subset\R$ or
the negative binomial distribution on $\Z_+\subset\R$, respectively. Note that
in the latter case $w$ is purely atomic, so that the remark after
\eqref{orth.B} concerning the possibility to remove the hyperplanes $x_i=x_j$
becomes really meaningful. We also use the explicit formulas \eqref{eq1} and
\eqref{eq2} for the squared norms. This leads to the following claim:

\begin{corollary}
For the $N$-variate Laguerre and Meixner symmetric polynomials we have
\begin{align}
(L_{\mu\mid N,b},L_{\nu\mid N,b})_N
&=\de_{\mu\nu}\cdot\prod_{i=1}^N(N-i+1)_{\nu_i}(N+b-i)_{\nu_i}\label{orth.C}\\
(M_{\mu\mid N,b,\xi},M_{\nu\mid N,b,\xi})_N
&=\de_{\mu\nu}\cdot\frac{\xi^{|\nu|}}{(1-\xi)^{2|\nu|}}
\prod_{i=1}^N(N-i+1)_{\nu_i}(N+b-i)_{\nu_i}\label{orth.D}
\end{align}
\end{corollary}

\section{The Laguerre and Meixner symmetric functions}\label{LaMe}

\subsection{Preliminaries on the algebra $\Sym$ of symmetric functions}
So far we were dealing with partitions $\nu$ of restricted length, but in what
follows we consider arbitrary partitions and we agree to identify them with
Young diagrams. The set of all partitions (=Young diagrams) will be denoted by
$\Y$. We denote by $\ell(\nu)$ the length of $\nu$, that is, the number of
nonzero coordinates (or nonzero rows).

\begin{definition} For every $N\ge1$, consider
the algebra morphism
\begin{equation*}
\pi_{N-1,N}\colon\Sym(N)\to\Sym(N-1)
\end{equation*}
defined by
\begin{equation*}
(\pi_{N-1,N}f)(x_1,\dots,x_{N-1})=f(x_1,\dots,x_{N-1},0).
\end{equation*}
Let $\Sym$ consist of all sequences $(f_N\in\Sym(N))$ such that, first,
$\pi_{N-1,N}f_N=f_{N-1}$ for each $N$ and, second, $\sup\deg f_N<\infty$. This
is an algebra under termwise operations; it is called the {\it algebra of
symmetric functions\/}. An element $f=(f_N)$ is said to be homogeneous of
degree $k$ if so are all $f_N$'s. The canonical morphism $\Sym\to\Sym(N)$
taking $f$ to $f_N$ will be denoted by $\pi_N$.
\end{definition}

The above definition follows \cite{Ma95} (only we are working over $\R$, not
$\Z$). Equivalently, one can say that $\Sym$ is the subalgebra in
$\R[[x_1,x_2,\dots]]$ formed by symmetric formal power series of bounded
degree, see \cite{Sa01}.

By the very definition, $\Sym$ is a graded algebra. There are three
distinguished systems of algebraically independent homogeneous generators of
$\Sym$: the {\it Newton power sums\/}
\begin{equation*}
p_k=\sum_{i=1}^\infty x_i^k, \qquad k=1,2,\dots,
\end{equation*}
the {\it elementary symmetric functions}
$$
e_k=\sum_{1\le i_1<\dots<i_k}x_{i_1}\dots x_{i_k}, \qquad k=1,2,\dots,
$$
and the {\it complete homogeneous symmetric functions}
$$
h_k=\sum_{1\le i_1\le\dots\le i_k}x_{i_1}\dots x_{i_k}, \qquad k=1,2,\dots\,.
$$
Thus, $\Sym$ may be identified with the algebra of polynomials in each of these
three systems of variables,
$$
\Sym=\R[p_1,p_2,\dots]=\R[e_1,e_2,\dots]=\R[h_1,h_2,\dots],
$$
with the understanding that $\deg p_k=\deg e_k=\deg h_k=k$.

The Schur polynomials are stable in the sense that
\begin{equation}\label{LaMe.G}
S_{\nu\mid N}\big|_{x_N=0}=S_{\nu\mid N-1},
\end{equation}
where by convention $S_{\nu\mid N}\equiv0$ if $N<\ell(\nu)$. This make it
possible to define their analogs in the algebra $\Sym$, the {\it Schur
symmetric functions\/}
$$
S_\nu=\varprojlim S_{\nu\mid N}, \qquad \nu\in\Y,
$$
characterized by
\begin{equation*}
\pi_N(S_\nu)=S_{\nu\mid N} \qquad \forall N.
\end{equation*}

The Schur functions $S_\nu$ form a homogeneous basis in $\Sym$. They are
expressed through $\{e_k\}$ and $\{h_k\}$ as follows
$$
S_\nu=\det[e_{\nu'_i-i+j}]=\det[h_{\nu_i-i+j}],
$$
where $\nu\,'$ is the transposed Young diagram and $e_0=h_0=1$,
$e_{-k}=h_{-k}=0$ for $k=1,2,\dots$; the order of determinants may be taken
arbitrarily provided it is large enough.

In the initial definition, the algebra $\Sym$ is tied to an infinite collection
of variables $x_1,x_2,\dots$, but often it is preferable to adopt a different
point of view and interpret symmetric functions as polynomials in $\{p_k\}$ or
$\{h_k\}$ or $\{e_k\}$, which creates extra degrees of freedom and leads to
useful realizations of the algebra $\Sym$, other than its initial realization
inside $\R[[x_1,x_2,\dots]]$.

An important example, which we substantially exploit below, is the {\it
supersymmetric realization\/} of $\Sym$ inside the algebra
$\R[[x_1,x_2,\dots;y_1,y_2,\dots]]$ of formal series in a doubly infinite
collection of variables. In this realization,
$$
p_k\,\to\, \sum_{i=1}^\infty x_i^k +(-1)^{k-1}\sum_{i=1}^\infty y_i^k, \qquad
k=1,2,\dots\,.
$$

Let $\si:\Sym\to\Sym$ denote the involutive linear map  given by
$\si(S_\nu)=S_{\nu'}$ for each $\nu\in\Y$. (In \cite{Ma95}, it is denoted by
$\om$.) This is an algebra automorphism interchanging $h_k$ with $e_k$ and
taking $p_k$ to $(-1)^{k-1}p_k$. In the supersymmetric realization, the
involution $\si$ is implemented by the automorphism of
$\R[[x_1,x_2,\dots;y_1,y_2,\dots]]$ interchanging $x_i$ with $y_i$ for all $i$,
whereas in the standard realization $\Sym\subset\R[x_1,x_2,\dots]$, such a
natural interpretation is lacking.

\subsection{The Laguerre symmetric functions}

Return to the expansion \eqref{Ndim.C} and examine the expression
\begin{equation*}
\prod_{i=1}^N\frac{(\nu_i+N-i)!(b)_{\nu_i+N-i}}{(\mu_i+N-i)!(b)_{\mu_i+N-i}} =
\prod_{i=1}^N\frac{n_i!(b)_{n_i}}{m_i!(b)_{m_i}}
\end{equation*}
for the coefficients in front of the Schur polynomials. Recall that
$\nu\in\Y(N)$ and $\mu\subseteq\nu$; as before, we also use the shorthand
notation $n_i=\nu_i+N-i$, $m_i=\mu_i+N-i$.

We are going to rewrite this expression as a product over the boxes of the skew
Young diagram $\nu/\mu$. We need an extra notation: given a box $\square=(i,j)$
(meaning that $i$ and $j$ are the row and column numbers of $\square$), the
difference $j-i$ is called the {\it content\/} of  $\square$ and is denoted as
$c(\square)$. The following is a generalization of the Pochhammer symbol:
\begin{equation*}
(z)_{\nu/\mu}=\prod_{\square\in\nu/\mu}(z+c(\square)), \qquad z\in\C.
\end{equation*}
In the particular case when $\mu=\varnothing$ and $\nu=(n)$ or $\nu=(1^n)$,
this gives $(z)_n$ or $z^{\down n}$, respectively.

\begin{lemma}\label{LaMe.1}
With the above notation,
\begin{equation}\label{eq3}
\prod_{i=1}^N\frac{n_i!(b)_{n_i}}{m_i!(b)_{m_i}}
=(N)_{\nu/\mu}(N+b-1)_{\nu/\mu}.
\end{equation}
\end{lemma}

\begin{proof}
The $i$th term entering the product in left-hand side equals
\begin{equation*}
(m_i+1)\dots n_i\cdot (m_i+b)\dots (n_i+b-1),
\end{equation*}
which is the same as the product in the right-hand side restricted to the boxes
entering the $i$th row of $\nu/\mu$.
\end{proof}

Despite the different origin of $N$ and $b$, the quantities $N$ and $N+b-1$
enter the right-hand side of \eqref{eq3} symmetrically. This observation is the
starting point for the next definition.

\begin{definition}\label{LaMe.3}
Let $z$ and $z'$ be complex parameters and $\Sym_\C=\Sym\otimes_\R\C$ denote
the complexification of the algebra $\Sym$. The {\it Laguerre symmetric
function\/} $\LL_\nu\in\Sym_\C$ with index $\nu\in\Y$ and parameters $(z,z')$
is defined by the following expansion in the basis of the Schur symmetric
functions:
\begin{equation}\label{eq5}
\LL_\nu=\sum_{\mu:\,\mu\subseteq\nu} (-1)^{|\nu|-|\mu|}
\frac{\dim\nu/\mu}{(|\nu|-|\mu|)!}\, (z)_{\nu/\mu}(z')_{\nu/\mu} S_\mu.
\end{equation}
That is, we formally rename $N$ and $N+b-1$ by $z$ and $z'$, and replace in
\eqref{Ndim.C} the Schur polynomials by the Schur symmetric functions.
\end{definition}

Clearly, $\LL_\nu$ is an inhomogeneous element of $\Sym$ of degree $|\nu|$,
with top degree homogeneous component equal to $S_\nu$. Let $\Sym^{\deg\le
n}\subset\Sym$ denote the subspace of elements of degree less or equal to $n$.
Clearly, for each natural $n$, the elements $\LL_\nu$ with $|\nu|\le n$ form a
basis in $\Sym^{\deg\le n}$. Thus, $\{\LL_\nu\}$ is a basis in $\Sym$
consistent with the canonical filtration of $\Sym$ determined by its
graduation.

Note that, since parameters $z$ and $z'$ enter \eqref{eq5} in a symmetric way,
$\LL_\nu$ is not affected by the transposition $z\leftrightarrow z'$.

One more remark is that
$$
\si(\LL_\nu)=\LL_{\nu'}\big|_{z\to-z,\, z'\to-z'}.
$$
In words: application of the involution $\si$ to $\LL_\nu$ results in
transposition of the index $\nu$ and multiplication of the parameters $z$ and
$z'$ by $-1$. This nice symmetry relation appears only on the level of the
algebra $\Sym$; in the finite-variate case it does not exist.

\subsection{Analytic continuation}

Let us start with a few claims which are obvious consequences of the results of
the previous two subsections.

Let $J_N\subset\Sym$ denote the kernel of the projection
$\pi_N:\Sym\to\Sym(N)$; $J_N$ is the ideal in $\Sym$ generated by elements
$e_k$ with indices $k>N$.

\begin{proposition}\label{prop1}
For every $N=1,2,\dots$, the Schur functions $S_\nu$ with $\ell(\nu)>N$ form a
basis in $J_N$, while the remaining Schur functions span a complement to $J_N$.
\end{proposition}

Recall that $\Sym^{\deg\le n}\subset\Sym$ denotes the finite-dimensional
subspace formed by elements of degree less or equal to $n$.

\begin{corollary}\label{cor1}
Fix an arbitrary $n=1,2,\dots$\,. For $N$ large enough, one has
$J_N\cap\Sym^{\deg\le n}=\{0\}$.
\end{corollary}

More precisely, the above relation holds starting from $N=n$.

The following proposition provides a characterization of the Laguerre symmetric
functions and explains the origin of their definition:

\begin{proposition}\label{prop2}
For any fixed $\nu\in\Y$, the Laguerre symmetric function $\LL_\nu$ can be
characterized as the only element of the algebra
$$
\Sym[z,z']=\Sym\otimes\C[z,z']
$$
such that for any natural $N$ and any $b>0$ one has
\begin{equation*}
\pi_N\left(\LL_\nu\big|_{z=N,\, z'=N+b-1}\right) =\begin{cases}L_{\nu\mid N,b},
& \text{\rm if $\ell(\nu)\le N$}\\
0, &\text{\rm if $\ell(\nu)>N$.}
\end{cases}
\end{equation*}
\end{proposition}

\begin{proof}
Recall that
\begin{equation*}
\pi_N\left(S_\nu\big|_{z=N,\, z'=N+b-1}\right) =\begin{cases}S_{\nu\mid N},
& \text{\rm if $\ell(\nu)\le N$}\\
0, &\text{\rm if $\ell(\nu)>N$.}
\end{cases}
\end{equation*}

Using this and comparing \eqref{eq5} and \eqref{Ndim.C} we see that $\LL_\nu$
has the required property.

Let us prove the uniqueness statement. It says that if an element $F\in
\Sym[z,z']$ is such that
$$
\pi_N\left(F\big|_{z=N,\, z'=N+b-1}\right)=0
$$
for all $N$ and $b$ as above, then $F=0$.

Fix $n$ so large that
$$
F\in\Sym^{\deg\le n}[z,z'].
$$
By virtue of Corollary \ref{cor1}, for large $N$, a stronger condition holds:
$$
F\big|_{z=N,\, z'=N+b-1}=0
$$
This implies $F\equiv0$, because $F$ is a (vector-valued) polynomial in
$(z,z')$, and any point set of the form
$$
\{(z,z')=(N,N+b-1): N=N_0,N_0+1,N_0+2,\dots, \quad b>0\}\subset\C^2,
$$
is a uniqueness set for polynomials in two variables.
\end{proof}

Informally, Proposition \ref{prop2} may be interpreted as follows:

\smallskip

{\it The Laguerre symmetric functions \eqref{eq5} are obtained from the
$N$-variate Laguerre symmetric polynomials \eqref{Ndim.C} by  analytic
continuation with respect to parameters $N$ and $b$.}

\smallskip

As we will see, Proposition \ref{prop2} makes it possible to prove various
algebraic relations involving the Laguerre symmetric functions using the
principle of analytic continuation of identities. The crucial fact is that, by
virtue of Lemma \ref{LaMe.1}, the dependence of \eqref{Ndim.C} in $N$ and $b$
is polynomial, which makes extrapolation from discrete values of parameter $N$
to the complex domain unambiguous.

\subsection{The Laguerre differential operator $\D^\LA:\Sym_\C\to\Sym_\C$}

A box $\square$ in a Young diagram $\nu$ is said to be a {\it corner box\/} if
the shape $\nu\setminus\square$ obtained by removing $\square$ from $\nu$ is
again a Young diagram. By $\nu^-$ we denote the set of all corner boxes in
$\nu$. For instance, if $\nu=(3,2,2)$ then $\nu^-$ comprises two corner boxes,
$\square=(1,3)$ and $\square=(3,2)$.

\begin{definition}\label{LaMe.4}
Introduce the {\it Laguerre operator\/} $\D^\LA\colon\Sym_\C\to\Sym_\C$ with
parameters $z,z'$ by setting
\begin{equation}\label{eq7}
\D^\LA
S_\nu=-|\nu|S_\nu+\sum_{\square\in\nu^-}(z+c(\square))(z'+c(\square))S_{\nu\setminus\square}
\end{equation}
\end{definition}

This formula is obtained from \eqref{Ndim.B} by formal substitution
$$
N\to z, \qquad N+b-1\to z'.
$$
Note that the resulting expression for $\D^\LA$ is symmetric under
$z\leftrightarrow z'$. Note also that
\begin{equation}\label{eq26}
\si\circ\D^\LA\circ\si=\D^\LA\big|_{z\to-z,\, z'\to-z'}.
\end{equation}

\begin{proposition}\label{prop5}
If $z=N=1,2,\dots$ and $z'=N+b-1$ with $b>0$, then the operator $\D^\LA$
preserves the ideal $J_N\subset\Sym$,  and its action on the quotient space
$\Sym/J_N=\Sym(N)$ coincides with that of the $N$-variate Laguerre operator
$D^\LA_N$ with parameter $b$.
\end{proposition}

\begin{proof}
Let us prove the first claim: $\D^\LA J_N\subseteq J_N$. Recall that $J_N$ is
spanned by the Schur functions $S_\nu$ with $\ell(\nu)\ge N+1$. Thus, it
suffices to check that $\D^\LA S_\nu\in J_N$ provided that $\ell(\nu)\ge N+1$.
By the very definition, $\D^\LA S_\nu$ is a linear combination of $S_\nu$ and
the elements of the form $S_{\nu\setminus\Box}$. We have
$\ell(\nu\setminus\Box)\ge N+1$ with the only exception when $\ell(\nu)=N+1$,
$\nu_{N+1}=1$, and $\Box=(N+1,1)$. But in this case the factor
$z+c(\Box)=N+c(\Box)$ in front of $S_{\nu\setminus\Box}$ vanishes, because
$c(\Box)=-N$.

The second claim is obvious, because the action of the operator $\D^\LA$ with
parameters $z=N$ and $z'=N+b-1$ on the Schur functions $S_\nu$ with
$\ell(\nu)\le N$ is exactly the same as the action of the operator $D_N$ on the
$N$-variate Schur polynomials $S_{\nu\mid N}=\pi_N(S_\nu)$.
\end{proof}

\begin{proposition}\label{prop8}
The Laguerre symmetric functions are eigenvectors of the operator\/
$\D^\LA${\rm:} we have\/ $\D^\LA \LL_\nu=-|\nu|\LL_\nu$ for every $\nu\in\Y$.
\end{proposition}

\begin{proof}
We argue as in the proof of Proposition \ref{prop2}. Both sides of the equality
in question are elements of a space of vector-valued polynomials,
$$
\Sym^{\deg\le |\nu|}[z,z']=\Sym^{\deg\le |\nu|}\otimes\C[z,z'].
$$
When $z=N$ and $z'=N+b-1$, the desired equality holds after factorization
modulo $J_N$: this follows from Proposition \ref{prop5}, because the similar
equality holds in the $N$-variable case, see \eqref{Ndim.B}. Next, the same
argument as above shows that the equality actually holds without factorization,
provided that $N$ is large enough. This suffices to conclude that the equality
holds for all complex values of $z$ and $z'$.
\end{proof}

Further, observe that any linear operator in an algebra of polynomials (with
finitely or countably many variables) can be represented as a formal
differential operator with polynomial coefficients. In the next theorem we
describe such a presentation for $\D^\LA$.

\begin{theorem}\label{thm1}
Upon the identification\/ $\Sym_\C=\C[e_1,e_2,\dots]$, the Laguerre operator\/
$\D^\LA:\Sym_\C\to\Sym_\C$ with parameters $z$ and $z'$ can be written as a
second order differential operator in variables $e_1,e_2,\dots$,
$$
\begin{aligned}
\D^\LA&=\sum_{n\ge1}\left(\sum_{k=0}^{n-1}(2n-1-2k)e_{2n-1-k}e_k\right)
\frac{\pd^2}{\pd e_n^2}\\
&+2\sum_{m>n\ge1}\left(\sum_{k=0}^{n-1}(m+n-1-2k)e_{m+n-1-k}e_k\right)
\frac{\pd^2}{\pd e_m\pd e_n}\\
&+\sum_{n=1}^\infty\big(-ne_n+(z-n+1)(z'-n+1)e_{n-1}\big)\frac{\pd}{\pd e_n}
\end{aligned}
$$
with the understanding that $e_0=1$.
\end{theorem}

First, we will prove two lemmas.

\begin{lemma}\label{lemma3}
Temporarily denote by $X$ the formal differential operator in the right-hand
side and regard $X$ as an operator acting in\/ $\Sym_\C$. If $z=N=1,2,\dots$,
then $X$ preserves the ideal $J_N\subset\Sym$.
\end{lemma}

\begin{proof}
Split $X$ into the sum $X_2+X_1$ of the second and first order terms. We use
the fact that $J_N$ is generated by the elements $e_n$ with $n\ge N+1$. This
means that $J_N$ is spanned by those monomials in the generators
$e_1,e_2,\dots$ that contain at least one letter $e_n$ with $n\ge N+1$. Take
any such monomial and apply to it the operator $X_2$. {}From the explicit form
of $X_2$ it is seen that if a letter $e_n$ disappears after application of a
differential monomial entering the formula, then another letter $e_{n'}$ with
index $n'\ge n$ comes from the coefficient of this monomial, so that the result
is again contained in $J_N$.

Obviously, the same happens after application of the term $\sum(-e_n)\pd/\pd
e_n$ in $X_1$.

Examine now the remaining first order term in $X_1$. Its application to $e_n$
lowers the index $n$ by $1$, so a possible problem arises when $n=N+1$. But
then the coefficient $(z-n+1)(z'-n+1)$ vanishes because, by the assumption, $z$
is specialized to $N$.
\end{proof}

Keep the notation $X$ for the above differential operator, fix $N=1,2,\dots$,
and denote by $X_N$ the differential operator in variables $e_1,\dots,e_N$ that
is obtained from operator $X$ by the following truncation procedure: first,
keep only terms which do not contain derivatives on variables $e_n$ with $n\ge
N+1$, next, put $e_{N+1}=e_{N+2}=\dots=0$ in the coefficients, and, finally,
specialize $z=N$ and $z'=N+b-1$.

\begin{lemma}\label{lemma4}
Identify\/ $\Sym(N)$ with the algebra\/ $\R[e_1,\dots,e_N]$. Then the operator
$D^\LA_N:\Sym(N)\to\Sym(N)$ with parameter $b>0$ coincides with the operator
$X_N$ just defined.
\end{lemma}

\begin{proof}
We start by recalling a well-known abstract formalism. Let $\mathcal A$ be a
commutative unital algebra and $\mathcal X:\mathcal A\to\mathcal A$ be a linear
operator. For $a\in\mathcal A$, denote by $M_a$ the operator of multiplication
by $a$. Let us say that $\mathcal X$ has order $\le k$ if its $(k+1)$-fold
commutator with operators of multiplication by arbitrary elements of the
algebra vanishes:
$$
[M_{a_1},[M_{a_2,},\dots [M_{a_{k+1}},\mathcal X]\dots]]=0, \qquad
a_1,\dots,a_{k+1}\in\mathcal A.
$$
If $X$ has order $\le0$ then $\mathcal X=M_a$, where $a=X1$. If $\mathcal X$
has order $\le k$ with $k\ge1$ and $a\in\mathcal A$ is arbitrary, then
$[M_a,\mathcal X]$ has order $\le k-1$. Using this, it is an easy exercise to
check that if $\mathcal A$ is generated by a sequence of elements
$a_1,a_2,\dots$, then any operator $\mathcal X$ of order $\le k$ is uniquely
determined by its action on monomials of degree $\le k$ in the generators.

Obviously, a differential operator of order $k$ has order $\le k$ in this
abstract sense. Consequently, the operator $D^\LA_N$ has order $\le 2$ in the
abstract sense, since it can be written as a second order differential operator
(the fact that the coefficients have singularities along the hyperplanes
$x_i=x_j$ is inessential here).

Return now to the equality in question, $D^\LA_N=X_N$. Both $D^\LA_N$ and $X_N$
are operators in $\R[e_1,\dots,e_N]$ of order $\le2$ in the abstract sense.
Therefore, it suffices to verify that they coincide on $1$, on the generators
$e_n$, and on quadratic elements $e_me_n$, where $m, n=1,\dots,N$.

The first assertion is obvious, as both operators annihilate $1$.

The second assertion means that
$$
D^\LA_N e_n=-e_n+(N-n+1)(N-n+b)e_{n-1}.
$$
This is clear from \eqref{Ndim.B}, because $e_n=S_{(1^n)}$.

The third assertion amounts to the equality
$$
\frac12\left(D^\LA_N(e_me_n)-(D^\LA_N e_m)e_n-e_m(D^\LA_N e_n)\right)=
\sum_{k,l}(k-l)e_ke_l,
$$
summed over couples $k>l$ such that $k+l=m+n-1$, $0\le l\le n-1$, and $k\le N$.
Here we apply the formula
$$
e_me_n=\sum S_\nu,
$$
where summation is over two-column diagrams $\nu$ such that $\nu\,'=(m+r,n-r)$,
where $0\le r\le n$ and $m+r\le N$. This allows us to apply formula
\eqref{Ndim.B}. Then we use the fact that if $\nu$ is a two-column diagram and
$\nu\,'=(p,q)$, then
$$
S_\nu=e_pe_q-e_{p+1}e_{q-1}
$$
with the understanding that $e_0=1$ and $e_{-1}=e_{N+1}=e_{N+2}=\dots=0$. Then
there are many cancellations and finally we get the desired result.
\end{proof}

\begin{proof}[Proof of the theorem]
Let $X$ be the same differential operator as above. We will prove that  $\D^\LA
F=X F$ for every $F\in\Sym$ by using the principle of analytic continuation.

As in the proof of Proposition \ref{prop2}, it is enough to establish the
equality
$$
\pi_N\left(\D^\LA\big|_{z=N,\,z'=N+b-1}F\right)
=\pi_N\left(X\big|_{z=N,\,z'=N+b-1}F\right).
$$

Set $F_N=\pi_N(F)$. By Proposition \ref{prop5}, the left-hand side equals
$D^\LA_N F_N$. By Lemma \ref{lemma3}, the right-hand side equals $X_N F_N$.
Finally, the Lemma \ref{lemma4} says that $D^\LA_N=X_N$, which concludes the
proof.
\end{proof}

\begin{corollary}\label{cor4}
Upon the identification\/ $\Sym_\C=\C[h_1,h_2,\dots]$, the Laguerre operator
$D^\LA:\Sym_\C\to\Sym_\C$ with parameters $z$ and $z'$ can be represented as a
second order differential operator in variables $h_1,h_2,\dots$,
$$
\begin{aligned}
\D^\LA&=\sum_{n\ge1}\left(\sum_{k=0}^{n-1}(2n-1-2k)h_{2n-1-k}h_k\right)
\frac{\pd^2}{\pd h_n^2}\\
&+2\sum_{m>n\ge1}\left(\sum_{k=0}^{n-1}(m+n-1-2k)h_{m+n-1-k}h_k\right)
\frac{\pd^2}{\pd h_m\pd h_n}\\
&+\sum_{n=1}^\infty\big(-h_n+(z+n-1)(z'+n-1)h_{n-1}\big)\frac{\pd}{\pd h_n}
\end{aligned}
$$
with the understanding that $h_0=1$.
\end{corollary}

\begin{proof}
Since the involution $\si:\Sym\to\Sym$ acts as $e_n\leftrightarrow h_n$, this
follows from \eqref{eq26}.
\end{proof}

\begin{remark}
In terms of the generators $p_1,p_2,\dots$, the Laguerre operator is given by
\begin{equation}\label{LaMe.C}
\begin{aligned}
\D^\LA &=\sum_{i=1}^\infty\left(-ip_i\frac{\pd}{\pd p_i}+(z+z')
(i+1)p_i\frac{\pd}{\pd p_{i+1}}\right) +zz'\frac{\pd}{\pd p_1}\\
&+\sum_{i,j=1}^\infty\left(ij p_{i+j-1}\frac{\pd}{\pd p_i}\frac{\pd}{\pd
p_j}+(i+j+1)p_ip_j\frac{\pd}{\pd p_{i+j+1}}\right).
\end{aligned}
\end{equation}
This formula is obtained by combining \eqref{eq7} with \cite[Lemma 6.3]{BO09};
it agrees with the formula given in \cite[Definition 3.5]{DH11}.
\end{remark}

Note that the Laguerre symmetric functions can be characterized by the
following two properties (cf. \cite{La91c}): first, $\LL_\nu$ differs from
$S_\nu$ by lower degree terms; second, $\LL_\nu$ is an eigenvector of the
operator $\D^\LA$ (necessarily, with the eigenvalue $-|\nu|$).

\subsection{Shifted symmetric functions}

We aim at applying the same principle of analytic continuation to constructing
symmetric functions analogs of the Meixner polynomials. Recall that our
starting point was the expansion formula \eqref{Ndim.C}. Its counterpart is
\eqref{Ndim.D}, so we need symmetric functions analogs of the factorial Schur
polynomials $S^\f_{\mu\mid N}$ entering \eqref{Ndim.D}. However, we cannot
directly imitate the above definition of the Schur functions via the Schur
polynomials because the factorial Schur polynomials do not share the stability
property \eqref{LaMe.G} of the ordinary Schur polynomials: Indeed, instead of
\eqref{LaMe.G}, the following relation holds
\begin{equation}
S^\f_{\nu\mid N}(x_1,\dots,x_N)\big|_{x_N=0}=S^\f_{\nu\mid
N-1}(x_1-1,\dots,x_{N-1}-1)
\end{equation}
(as above, we agree that $S^\f_{\nu\mid N}\equiv0$ when $\ell(\nu)>N$).

This circumstance forces us to take a roundabout way. Following \cite{OO97a},
say that a polynomial $f(y_1,\dots,y_N)$ is {\it shifted symmetric\/} if it
becomes symmetric in new variables $x_i:=y_i+N-i$. Let $\Sym^*(N)$ denote the
algebra of shifted symmetric polynomials in $N$ variables and let
$\Sym(N)\to\Sym^*(N)$ be the isomorphism determined by the change of variables
$x\to y$. The image in $\Sym^*(N)$ of $S^\f_{\nu\mid N}$ under this isomorphism
is called the {\it shifted Schur polynomial\/} with index $\nu$ and is denoted
by $S^*_{\nu\mid N}$:
\begin{equation*}
S^*_{\nu\mid N}(y_1,\dots,y_N)=S^\f_{\nu\mid N}(y_1+N-1,y_2+N-2,\dots,y_N).
\end{equation*}
Under the passage from factorial to shifted Schur polynomials the conventional
symmetry is lost but stability is recovered: the shifted Schur polynomials
enjoy exactly the same property as the ordinary ones, that is
\begin{equation*}
S^*_{\nu\mid N}\big|_{y_N=0}=S^*_{\nu\mid N-1},
\end{equation*}
see \cite{OO97a}.

Now we proceed in analogy with the definition of the algebra $\Sym$: Consider
the projective limit of the algebras $\Sym^*(N)$ taken with respect to
projections $\Sym^*(N)\to\Sym^*(N-1)$ defined by specializing the last variable
to $0$, and then take the subalgebra $\Sym^*$ formed by elements of bounded
degree. The fact that projections $\Sym^*(N)\to\Sym^*(N-1)$ respect the
filtration by the conventional degree of polynomials allows one to equip
$\Sym^*$ with a structure of filtered algebra. Note that the associated graded
algebra $\operatorname{gr}\Sym^*$ is canonically isomorphic to $\Sym$.

The algebra $\Sym^*$ is called the {\it algebra of shifted symmetric
functions\/} \cite{OO97a}. It is freely generated by $1$ and elements
\begin{equation*}
p^*_k(y_1,y_2,\dots)=\sum_{i=1}^\infty [(y_i-i+\tfrac12)^k-(-i+\tfrac12)^k],
\qquad k=1,2,\dots\,.
\end{equation*}

Further, the stability property of the $N$-variate shifted Schur polynomials
$S^*_{\nu\mid N}$ allows one to define their limits as $N\to\infty$. These are
certain elements $S^*_\nu\in\Sym^*$ called the {\it shifted Schur functions\/}
\cite{OO97a}. In the next subsection we explain how to convert them to ordinary
symmetric functions.

\subsection{The algebra $\A$ of polynomial functions on $\Y$}

For more details about the material of this subsection, see \cite{ORV03} and
\cite{IO03}.

\begin{definition}\label{def1}
We will need the notion of {\it modified Frobenius coordinates\/} of a
diagram $\la\in\Y$. This is a double collection
$(a;b)=(a_1,\dots,a_d;b_1,\dots,b_d)$ of half-integers, where $d$ stands for
the number of diagonal boxes in $\la$, $a_i=\la_i-i+\frac12$ equals the number
of boxes in the $i$th row of $\la$ plus one-half, and $b_i$ is the same
quantity for transposed diagram $\la'$.
\end{definition}

For instance, if $\la=(3,2,2)$ then $(a;b)=(2\frac12, \frac12; 2\frac12,
1\frac12)$.

Note that
\begin{equation*}
\sum_{i=1}^d (a_i+b_i)=|\la|.
\end{equation*}
The transposition map $\la\mapsto\la'$ switches $a$ and $b$. Thus, in the
coordinates $(a;b)$, rows and columns receive equal rights.

The notion of modified Frobenius coordinates has been suggested in \cite{VK81};
it differs from conventional Frobenius coordinates \cite{Ma95} by the
additional terms $\frac12$, which makes some formulas more symmetric. An
important example is the following nice identity.

\begin{lemma}
For any Young diagram  $\la=(a;b)\in\Y$ the following identity holds
\begin{equation}\label{pol.A}
\prod_{i=1}^\infty \frac{u+i-\tfrac12}{u-\la_i+i-\tfrac12}
=\prod_{i=1}^d\frac{u+b_i}{u-a_i}.
\end{equation}
\end{lemma}

Note that the product in the left-hand side is actually finite, because
$\la_i=0$ for all $i$ large enough, so that the numerator and denominator
cancel out.

\begin{proof}

See \cite[Proposition 1.2]{IO03}.

\end{proof}

Observe that \eqref{pol.A} is a rational function in $u$ taking value $1$ at
$u=\infty$. Hence its logarithm is well defined in a neighborhood $u=\infty$.
Expanding it into the Taylor series at $u=\infty$ with respect to the variable
$1/u$ we get the identity
\begin{equation}\label{pol.B}
p^*_k(\la_1,\la_2,\dots)=p_k(a;-b):=\sum_{i=1}^d[a_i^k -(-b_i)^k], \quad
k=1,2,\dots
\end{equation}
The expression in the right-hand side is the $k$th {\it supersymmetric\/} power
sum in variables $a=(a_i)$ and $-b=(-b_i)$. Abou the notion of supersymmetric
functions, see e.g. \cite[\S I.3, Ex. 23]{Ma95}.

\begin{definition}
Let $\A$ be the algebra of functions on $\Y$ generated over $\R$ by $1$ and the
functions \eqref{pol.B}. Elements of $\A$ are called {\it polynomial
functions\/} on $\Y$, see \cite{KO94}.
\end{definition}

It is readily verified that the functions \eqref{pol.B} are algebraically
independent. This makes it possible to define two algebra isomorphisms,
\begin{equation}\label{eq32}
\Sym^*\to\A \quad \textrm{and} \quad \Sym\to\A.
\end{equation}
The first one identifies the generators $p^*_k\in\Sym^*$ with the functions
\eqref{pol.B}, while the second one does the same with the generators
$p_k\in\Sym$. In words, the isomorphisms $\Sym^*\to\A\leftarrow\Sym$ mean that
shifted symmetric functions in the row coordinates $\la_i$ of Young diagrams
$\la\in\Y$ are the same as supersymmetric functions in $(a;-b)$, where
$a=(a_i)$ and $b=(b_i)$ are the modified Frobenius coordinates of $\la$.

This induces an isomorphism $\Sym^*\to\Sym$ sending $p^*_k$ to $p_k$ for every
$k=1,2,\dots$. This isomorphism seems to be the most natural way to lift the
canonical isomorphism $\operatorname{gr}\Sym^*=\Sym$ to an algebra isomorphism
$\Sym^*\to\Sym$.

The image under $\Sym^*\to\Sym$ of the shifted Schur functions $S^*_\nu$ are
some elements of $\Sym$ called the {\it Frobenius--Schur functions\/}
$\FS_\nu$. They are studied in detail in \cite{ORV03}.

The Frobenius--Schur functions have many nice properties. Here we only note the
following one, which directly follows from the very definition:
\begin{equation*}
\FS_\nu=S_\nu\, +\, \textrm{lower degree terms}.
\end{equation*}
This implies, in particular, that the elements $\FS_\nu$ with index $\nu$
ranging over $\Y$ form a basis in $\Sym$.

Let us agree to identify elements $F\in\Sym$ with the corresponding elements of
$\A$ and write them as functions $F(\la)$, where the argument $\la$ ranges over
$\Y$. In particular, we will employ the functions $\FS_\nu(\la)$.

Given $\la\in\Y(N)$, set
\begin{equation}\label{eq4}
x_i=\la_i+N-i, \qquad 1\le i\le N.
\end{equation}
The correspondence
$$
\la\mapsto (x_1,\dots,x_N)
$$
is a bijection between $\Y(N)$ and the set
\begin{equation*}
\Z_{+,\ord}^N:=\R^N_\ord\cap\Z_+^N=\{(l_1,\dots,l_N)\in\Z_+^N\colon
l_1>\dots>l_N\}.
\end{equation*}

\begin{definition}
For every $N=1,2,\dots$, define an algebra morphism
$$
\pi'_N:\Sym \to \Sym(N)
$$
by setting
$$
(\pi'_N(p_k))(x_1,\dots,x_N)=\sum_{i=1}^N[(x_i-N+\tfrac12)^k-(-i+\tfrac12)^k],
\quad k=1,2,\dots\,.
$$
\end{definition}

Since the right-hand side is a symmetric polynomial, the definition makes
sense.

The motivation of this definition is the following: For any $F\in\Sym$ and any
$\la\in\Y(N)\subset\Y$ one has
$$
F(\la)=\pi'_N(F)(x_1,\dots,x_N), \qquad (x_1,\dots,x_N)\leftrightarrow\la.
$$

In what follows, the maps $\pi'_N$ replace the truncation maps $\pi_N$.

\begin{remark}
One may view the maps $\pi'_N$ as a {\it deformation\/} of the maps $\pi_N$. To
see this, introduce the morphisms $\pi^{(\epsi)}_N:\Sym\to\Sym(N)$ by setting
\begin{align*}
(\pi^{(\epsi)}_N(p_k))(x_1,\dots,x_N)
&=\epsi^k(\pi'_N(p_k))(\epsi^{-1}x_1,\dots,\epsi^{-1}x_N)\\
&=\sum_{i=1}^N[(x_i +\epsi(-N+\tfrac12))^k-\epsi^k(-i+\tfrac12)^k].
\end{align*}
For $\epsi=1$ this coincides with  $\pi'_N$, and in the limit as $\epsi\to0$ we
get $\pi_N$.
\end{remark}

Now, the following proposition summarizes the discussion in this and preceding
subsections:

\begin{proposition}\label{prop7}
Let us identify the algebra\/ $\Sym$ of symmetric functions with the algebra
$\A$ of polynomial functions on $\Y$, as explained above.

{\rm(i)} The maps $\pi'_N$ relate the Frobenius-Schur functions
$\FS_\nu\in\Sym$ to the factorial Schur polynomials $S^\f_{\nu\mid N}$. Namely,
\begin{equation*}
\pi'_N(\FS_\nu)=\begin{cases} S^\f_{\nu\mid N}, &\text{\rm if $\ell(\nu)\le
N$}\\
0, &\text{\rm if $\ell(\nu)>N$}.
\end{cases}
\end{equation*}

{\rm(ii)} Let $J'_N\subset\Sym$ denote the kernel of\/ $\pi'_N$. The
Frobenius-Schur functions $\FS_\nu$ with $\ell(\nu)>N$ form a basis in $J'_N$,
which implies that
$$
\bigcap_{N=1}^\infty J'_N=\{0\}.
$$
\end{proposition}

\subsection{The Meixner symmetric functions}

Now we follow the same line of arguments as in the Laguerre case.

\begin{definition}[cf. Definition \ref{LaMe.3}]
The {\it Meixner symmetric function\/} $\MM_\nu$ with index $\nu\in\Y$ and
complex parameters $(z,z',\xi)$ is given by the following expansion in the
Frobenius--Schur symmetric functions:
\begin{equation}\label{eq6}
\begin{aligned}
\MM_\nu&=\sum_{\mu:\,\mu\subseteq\nu}
(-1)^{|\nu|-|\mu|}\left(\frac\xi{1-\xi}\right)^{|\nu|-|\mu|}
\frac{\dim\nu/\mu}{(|\nu|-|\mu|)!}\\
&\times \prod_{\square\in\nu/\mu}(z+c(\square))(z'+c(\square))\cdot \FS_\mu.
\end{aligned}
\end{equation}
\end{definition}

Like the Laguerre function $\LL_\nu$, the function $\MM_\nu$ is an
inhomogeneous element of $\Sym_\C$ of degree $|\nu|$, and its top degree
homogeneous component coincides with the Schur function $S_\nu$. It follows
that the elements $\MM_\nu$ form a basis in $\Sym_\C$ consistent with the
canonical filtration of $\Sym_\C$ determined by its graduation.

Let
$$
\C[z,z',\xi,(1-\xi)^{-1}]]\subset\C(\xi)[z,z']
$$
stand for the algebra of polynomials in $z$ and $z'$ whose coefficients are
rational functions in $\xi$ with possible poles only at $\xi=1$. The Meixner
symmetric functions may also be viewed as elements of the algebra
$$
\Sym\otimes \C[z,z',\xi,(1-\xi)^{-1}].
$$

The above definition is justified by the following proposition, which is
parallel to Proposition \ref{prop2}:

\begin{proposition}\label{prop6}
For any fixed $\nu\in\Y$, the Meixner symmetric function $\MM_\nu$ can be
characterized as the only element of the algebra
$$
\Sym[z,z',\xi,(1-\xi)^{-1}]=\Sym\otimes \C[z,z',\xi,(1-\xi)^{-1}]
$$
such that for any natural $N$, any $b>0$, and any $\xi\in(0,1)$ one has
\begin{equation*}
\pi'_N\left(\MM_\nu\big|_{z=N,\, z'=N+b-1}\right) =\begin{cases}M_{\nu\mid
N,b,\xi},
& \text{\rm if $\ell(\nu)\le N$}\\
0, &\text{\rm if $\ell(\nu)>N$.}
\end{cases}
\end{equation*}
\end{proposition}

\begin{proof}
The argument is the same as in Proposition \ref{prop2}. We use Proposition
\ref{prop7} and compare the expansions \eqref{eq6} and \eqref{Ndim.D}. We also
use the evident fact that any point set of the form
$$
\{(z,z',\xi)=(N,N+b-1,\xi): N=N_0,N_0+1,\dots, \, b>0, \,
\xi\in(0,1)\}\subset\C^3
$$
is a uniqueness set for the algebra $\C[z,z',\xi,(1-\xi)^{-1}]$: If an element
of this algebra vanishes on such a set, then $F$ equals $0$.
\end{proof}

\subsection{The Meixner difference operator $\D^\ME:\Sym_\C\to\Sym_\C$}

\begin{definition}[cf. Definition \ref{LaMe.4}]
We define the {\it Meixner operator\/}
$$
\D^\ME\colon\Sym_\C\to\Sym_\C
$$
with complex parameters $(z, z',\xi)$, $\xi\ne1$, by its action on the
Frobenius--Schur functions:
\begin{equation}\label{eq8}
\D^\ME \FS_\nu=-|\nu|\FS_\nu+\frac\xi{1-\xi}
\sum_{\square\in\nu^-}(z+c(\square))(z'+c(\square))\FS_{\nu\setminus\square}
\end{equation}
\end{definition}

Recall that $\nu^-\subset\nu$ denotes the subset of corner boxes in $\nu$.

\begin{proposition}[cf. Proposition \ref{prop8}]\label{prop20}
The Meixner symmetric functions are eigenvectors of the operator $\D^\ME${\rm:}
we have\/ $\D^\ME \MM_\nu=-|\nu|\MM_\nu$ for every $\nu\in\Y$.
\end{proposition}

\begin{proof}
We apply the principle of analytic continuation of identities. When $z=N$ and
$z'=N+b-1$, the desired equality holds because of \eqref{Ndim.F}. Since the
operator $\D^\ME$ depends polynomially on $z$ and $z'$, we are done.
\end{proof}

Recall that the Meixner operator $D^\ME_N:\Sym(N)\to\Sym(N)$ can be realized as
a partial difference operator in $N$ variables (Proposition \ref{prop9}). Here
is a generalization to the operator $\D^\ME:\Sym_\C\to\Sym_\C$:

\begin{proposition}\label{prop10}
Upon the identification $\Sym=\A$ the Meixner operator $\D^\ME$ is implemented
by the following operator in the space of functions on\/ $\Y$, which can be
written in two equivalent forms
\begin{gather*}
\D^\ME f(\la)=\sum_{\square\in\la^+}A(\la,\square)f(\la\cup\square) +
\sum_{\square\in\la^-}B(\la,\square)f(\la\setminus\square)-C(\la)f(\la)\\
=\sum_{\square\in\la^+}A(\la,\square)(f(\la\cup\square)-f(\la)) +
\sum_{\square\in\la^-}B(\la,\square)(f(\la\setminus\square)-f(\la)),
\end{gather*}
where
\begin{equation*}
\begin{aligned}
A(\la,\square)&=\frac\xi{1-\xi}(z+c(\square))(z'+c(\square))
\frac{\dim(\la\cup\square)}{(|\la|+1)\dim\la}, \quad
\square\in\la^+,\\
B(\la,\square)&=\frac1{1-\xi}\sum_{\square\in\la^-}
\frac{|\la|\dim(\la\setminus\square)}{\dim\la}, \quad \square\in\la^-,\\
C(\la)&=\frac1{1-\xi}((1+\xi)|\la|+\xi zz').
\end{aligned}
\end{equation*}
\end{proposition}

\begin{proof}
As functions of the parameters, the coefficients $A$, $B$, and $C$ belong to
the algebra $\C[z,z',\xi,(1-\xi)^{-1}]$. Using the principle of analytic
continuation as in Proposition \ref{prop2}, this enables us to reduce the claim
to the Lemma \ref{lemma1} below. In that lemma, we denote by $\D^{(1)}$ and
$\D^{(2)}$ the above two expressions for the operator $\D^\ME$ with parameters
$(z,z')$ specialized to $(N,N+b-1)$. We assume that $\la$ ranges over $\Y(N)$,
the subset of Young diagrams with at most $N$ nonzero rows. We also employ the
correspondence \eqref{eq4} to identify diagrams $\la\in\Y(N)$ with the
corresponding vectors $(x_1,\dots,x_N)\in\Z^N_{+,\ord}$.
\end{proof}

\begin{lemma}\label{lemma1}
Under the above hypotheses,  both\/ $\D^{(1)}$ and $\D^{(2)}$ preserve the
subspace of functions supported by the subset\/ $\Y(N)\subset\Y$, and their
restrictions to $\Y(N)$ coincide with the two equivalent forms of the
$N$-variate Meixner difference operator $D^\ME_N$ given in  Proposition
\ref{prop9}.
\end{lemma}

\begin{proof}
If $\la\in\Y(N)$ then the diagram $\la\cup\square$ can fall outside $\Y(N)$
only if $\ell(\la)=N$ and $\square=(N+1,1)$, which entails $c(\square)=-N$.
However, then $z+c(\square)$ vanishes for $z=N$, which implies vanishing of the
coefficient $A(\la,\square)$ in front of the term $f(\la\cup\square)$. This
shows that if $z=N$, then $\D^\ME$ can be restricted to $\Y(N)$.

Compare now the above formulas with \eqref{eq27} and \eqref{Ndim.K}, keeping in
mind the correspondence $\Y(N)\leftrightarrow\Z^N_{+,\ord}$.

First of all, it is readily verified that $C(x)=C(\la)$. Next, observe that the
shift $x\to x+\epsi$ amounts to appending a new box $\square$ to the $i$th row
of $\la$; likewise the shift $x\to x-\epsi_i$ amounts to removing the last box
$\square$ in the $i$th row. Such an operation is forbidden precisely in the
case when $x+\epsi_i$ or $x-\epsi_i$ falls outside $\Z^N_{+,\,\ord}$, and then
we know that the coefficient $A_i(x)$ or $B_i(x)$ automatically vanishes.
Therefore, it remains to check that if appending/removing a box $\square$ in
the $i$th row is possible, then
\begin{equation*}
A(\la,\square)=A_i(x), \qquad B(\la,\square)=B_i(x).
\end{equation*}
To see this, one applies the classical Frobenius' dimension formula \cite[\S
I.7, Ex. 6]{Ma95}, which in our notation looks as follows
\begin{equation*}
\frac{\dim\la}{|\la|!}=\frac{V_N(x)}{\prod_{i=1}^N x_i!}
\end{equation*}
\end{proof}

\begin{remark}
The fact that the two expressions given in Proposition \ref{prop10} give the
same result is equivalent to the relation
\begin{equation*}
\sum_{\square\in\la^+}A(\la,\square) +
\sum_{\square\in\la^-}B(\la,\square)=C(\la) \qquad\forall\la\in\Y.
\end{equation*}
Excluding from this relation the parameters $z$, $z'$, and $\xi$ reduces it to
a system of combinatorial identities involving only contents of boxes and
ratios of dimensions. These identities first appeared in Kerov's paper
\cite{Ke00}; see also \cite[Proposition 5.1]{Ol10a}.
\end{remark}

\subsection{The autoduality property of the Meixner symmetric functions}

The classical univariate Meixner polynomials are {\it autodual\/} in the sense
that, in an appropriate standardization, they are symmetric with respect to
transposition of the index and the argument, which both range over $\Z_+$. The
similar autoduality property holds for the Meixner symmetric functions realized
as functions on $\Y$ under the identification $\Sym=\A$:

\begin{proposition}
Change the standardization of\/ $\MM_\nu$ by setting
$$
\MM_\nu=\,\left(\frac{\xi}{1-\xi}\right)^{|\nu|} \frac{\dim\nu}{|\nu|!}
\prod_{\Box\in\nu}(z+c(\Box))(z'+c(\Box))\MM'_\nu,
$$
Then the following autoduality relation holds
$$
\MM'_\nu(\la)=\MM'_\la(\nu), \qquad \nu,\,\la\in\Y.
$$
\end{proposition}

\begin{proof}
There is a simple expression for the functions $\FS_\mu(\,\cdot\,)$ on $\Y$:
\begin{equation}\label{eq17}
\FS_\mu(\la)=\begin{cases}\dfrac{|\la|!}{(|\la|-|\mu|)!}\dfrac{\dim\la/\mu}{\dim\la},
& \textrm{if $\mu\subseteq\la$}\\
0, & \textrm{otherwise},
\end{cases} ,
\end{equation}
where $\la$ ranges over $\Y$; see \cite{ORV03}. Together with \eqref{eq6} this
gives
\begin{multline*}
\MM'_\nu(\la)
=\sum_{\mu\subseteq(\nu\cap\la)}(-1)^{|\mu|}\left(\frac{1-\xi}{\xi}\right)^{|\mu|}
\frac{|\nu|!|\la|!}{(|\nu|-|\mu|)!(|\la|-|\mu|)!}\\
\times\frac{\dim\nu/\mu\,
\dim\la/\mu}{\dim\nu\,\dim\la}\prod_{\Box\in\mu}\frac1{(z+c(\Box))(z'+c(\Box))}.
\end{multline*}
Clearly, this expression is symmetric under $\nu\leftrightarrow\la$.
\end{proof}

\subsection{Limit transition Meixner $\to$ Laguerre}

\begin{definition}
Introduce the following notation: Let $\epsi\in\C\setminus\{0\}$. Denote by
$G:\Sym\to\Sym$ the operator multiplying every homogeneous element by its
degree, and let $\epsi^{-G}:\Sym\to\Sym$ be the operator acting in the $m$th
homogeneous component of $\Sym$ as multiplication by $\epsi^{-m}$, for each
$m\in\Z_+$. Next, set
$$
\FS^{(\epsi)}_\nu=\epsi^{|\nu|}\epsi^{-G}\FS_\nu, \qquad \nu\in\Y.
$$
Finally, set
$$
\MM^{(\epsi)}_\nu=\epsi^{|\nu|} \epsi^{-G}\MM_\nu,
$$
where ``$\MM_\nu$'' in the right-hand side denotes the Meixner symmetric
function with parameters $(z,z')\in\C^2$ and $\xi=1-\epsi$.
\end{definition}

\begin{proposition}[cf. Proposition \ref{class.1}]
If $\nu$ and $(z,z')$ are fixed while $\epsi\to0$, then
\begin{equation}\label{eq18}
\FS^{(\epsi)}_\nu\to S_\nu
\end{equation}
and
\begin{equation}\label{eq19}
\MM^{(\epsi)}_\nu\to\LL_\nu,
\end{equation}
where convergence holds in the finite-dimensional subspace\/
$\Sym^{\deg\le|\nu|}\subset\Sym$.
\end{proposition}

\begin{proof}
The limit relation \eqref{eq18} is obvious, because the top degree homogeneous
component of $\FS_\nu$ coincides with $S_\nu$.

To prove \eqref{eq19}, observe that the expansions \eqref{eq5} and \eqref{eq6}
have the form
\begin{equation*}
\LL_\nu=\sum_{\mu:\,\mu\subseteq\nu} c_{\nu\mu}S_\mu, \quad
\MM_\nu=\sum_{\mu:\,\mu\subseteq\nu}
\left(\frac{1-\epsi}\epsi\right)^{|\nu|-|\mu|}c_{\nu\mu}\FS_\mu,
\end{equation*}
where the coefficients $c_{\nu\mu}$ do not depend on $\xi$. The second relation
implies
$$
\MM^{(\epsi)}_\nu=\sum_{\mu:\,\mu\subseteq\nu}
(1-\epsi)^{|\nu|-|\mu|}c_{\nu\mu}\FS^{(\epsi)}_\mu,
$$
which reduces \eqref{eq19} to \eqref{eq18}.
\end{proof}

Set
$$
\D^{\ME,\epsi}=\epsi^{-G}\circ\D^\ME\circ\epsi^G,
$$
where ``$\D^\ME$'' in the right-had side is the Meixner operator with
parameters $(z,z')\in\C^2$ and $\xi=1-\epsi$.

\begin{corollary}[cf. Remark \ref{rem1} (ii)]
We have
$$
\lim_{\epsi\to0}\D^{\ME,\epsi}=\D^\LA.
$$
\end{corollary}

The limit has an evident meaning because the operators in question preserve the
canonical filtration in $\Sym_\C$.

\begin{proof}
Recall that for any $\nu\in\Y$
$$
\D^\LA\LL_\nu=-|\nu|\LL_\nu, \qquad \D^\ME\MM_\nu=-|\nu|\MM_\nu.
$$
The last relation implies
$$
\D^{\ME,\epsi}\,\MM^{(\epsi)}_\nu=-|\nu|\MM^{(\epsi)}_\nu.
$$
Since $\MM^{(\epsi)}_\nu\to\LL_\nu$, this concludes the proof.

Alternatively, one may use Propositions \ref{prop8} and \ref{prop20}, and
formula \eqref{eq19}.
\end{proof}

\section{Orthogonality}

\subsection{Formal moment functionals on $\Sym$}\label{orth}

As above, in this subsection, $(z,z',\xi)$ is an arbitrary triple of complex
parameters with $\xi\ne1$.

\begin{definition}\label{def2}
Introduce linear functionals
\begin{equation*}
\varphi^\LA=\varphi^\LA_\zz:\Sym_\C\to\C, \qquad
\varphi^\ME=\varphi^\ME_\zxi:\Sym_\C\to\C
\end{equation*}
by setting

\begin{equation*}
\varphi^\LA(\LL_\nu)=\de_{\nu\varnothing}, \qquad
\varphi^\ME(\MM_\nu)=\de_{\nu\varnothing}.
\end{equation*}
for any $\nu\in\Y$. Since $\D^\LA\LL_\nu=-|\nu|\LL_\nu$ and
$\D^\ME\,\MM_\nu=-|\nu|\MM_\nu$, the above definition is equivalent to saying
that the functionals vanish on the range of the operators $\D^\LA$ and
$\D^\ME$, respectively, and equal $1$ on the unity element of $\Sym_\C$.

\end{definition}

\begin{proposition}\label{prop19}
We have
\begin{gather}
\varphi^\LA(S_\nu)=(z)_\nu(z')_\nu\frac{\dim\nu}{|\nu|!} \label{eq12}\\
\varphi^\ME(\FS_\nu)=\left(\frac{\xi}{1-\xi}\right)^{|\nu|}(z)_\nu(z')_\nu\frac{\dim\nu}{|\nu|!}
\label{eq13}
\end{gather}
\end{proposition}
Here $\de_{\nu\varnothing}$ is Kronecker's delta, which is equal to $0$ for
every $\nu\in\Y$ distinct from $\varnothing$, and to $1$ for $\nu=\varnothing$.
Recall that $\LL_\varnothing=\MM_\varnothing=1$.

\begin{proof}
Examine the Laguerre case. Since $\varphi^\LA$ vanishes on the range of
$\D^\LA$, \eqref{eq7} implies
\begin{equation*}
\varphi^{\,\LA}(S_\nu)
=\frac1{|\nu|}\sum_{\square\in\nu^-}(z+c(\square)(z'+c(\square))
\varphi^{\,\LA}(S_{\nu\setminus\square}), \qquad \nu\ne\varnothing.
\end{equation*}
Iterate and use the fact that, by the very definition of the quantity
$\dim\nu$, it equals the number of all possible ways of successive removals of
squares from $\nu$ ending at $\varnothing$. This leads to the desired formula.

The same argument works in the Meixner case; here we use \eqref{eq8}.
\end{proof}

Consider the structure constants for multiplication of symmetric functions in
the bases $\{\LL_\nu\}$ and $\{\MM_\nu\}$:
$$
\LL_\mu\LL_\nu=\sum_\la c^\LA_{\la\mu\nu}\LL_\la, \qquad
\MM_\mu\MM_\nu=\sum_\la c^\ME_{\la\mu\nu}\MM_\la.
$$
For fixed $(\la,\mu,\nu)$, the structure constants are functions in the
parameters:
$$
c^\LA_{\la\mu\nu}=c^\LA_{\la\mu\nu}(z,z'), \qquad
c^\ME_{\la\mu\nu}=c^\ME_{\la\mu\nu}(z,z',\xi).
$$

The next lemma is used in the theorem below.

\begin{lemma}\label{lemma2}
The Laguerre and Meixner structure constants $c^\LA_{\la\mu\nu}(z,z')$ and
$c^\ME_{\la\mu\nu}(z,z',\xi)$ belong to the algebras $\C[z,z']$ and
$\C[z,z',\xi,(1-\xi)^{-1}]$, respectively.
\end{lemma}

\begin{proof}
By \eqref{eq5}, the transition matrix between the bases of the Laguerre
symmetric functions and the Schur functions is unitriangular with respect to
the natural partial order on the index set $\Y$ given by inclusion of Young
diagrams, and the matrix entries are polynomials in $(z,z')$. Therefore, the
same holds for the inverse matrix. This shows that the Laguerre structure
constants are polynomial in $(z,z')$.

The same argument works for the Meixner structure constants as well; here we
use \eqref{eq6}.
\end{proof}

Note that top degree structure constants (that is, those with
$|\la|=|\mu|+|\nu|$) do not depend on the parameters and coincide with the the
Littlewood-Richardson coefficients, which are the structure constants in the
basis of Schur functions.

\begin{theorem}\label{thm2}
For any $\mu,\nu\in\Y$,
\begin{gather}
\varphi^\LA(\LL_\mu\LL_\nu)=\de_{\mu\nu}\cdot(z)_\nu(z')_\nu, \label{eq10}\\
\varphi^\ME(\MM_\mu\MM_\nu)=\de_{\mu\nu}\cdot\frac{\xi^{|\nu|}}{(1-\xi)^{2|\nu|}}(z)_\nu(z')_\nu.
\label{eq11}
\end{gather}
\end{theorem}

\begin{proof}
By the very definition of the linear functionals $\varphi^\LA$ and
$\varphi^\ME$,
$$
\varphi^\LA(\LL_\mu\LL_\nu)=c^\LA_{\varnothing\mu\nu}, \qquad
\varphi^\ME(\MM_\mu\MM_\nu)=c^\ME_{\varnothing\mu\nu}.
$$
Thus \eqref{eq10} and \eqref{eq11} are equivalent to
$$
c^\LA_{\varnothing\mu\nu}=\de_{\mu\nu}\cdot(z)_\nu(z')_\nu, \qquad
c^\ME_{\varnothing\mu\nu}=\de_{\mu\nu}\cdot\frac{\xi^{|\nu|}}
{(1-\xi)^{2|\nu|}}(z)_\nu(z')_\nu.
$$

By Lemma \ref{lemma2}, the left-hand sides of the equalities in questions are
elements of the algebras $\C[z,z']$ and $\C[z,z',\xi,(1-\xi)^{-1}]$,
respectively. Obviously, the same holds for the right-hand sides as well.
Applying the principle of analytic continuation, we reduce the problem to the
case when $(z,z')=(N,N+b-1)$ and, in the Meixner case, $\xi\in(0,1)$. But then
our equalities are reduced to the orthogonality relations for the $N$-variate
Laguerre and Meixner polynomials, see \eqref{orth.C} and \eqref{orth.D}.
\end{proof}

Note that the similar functionals on $\Sym(N)$ coincide with expectation with
respect to the Laguerre or Meixner version of the weight measure $w_N$ (see
\eqref{eq14}). In particular, in the case $N=1$ these are expectations with
respect to the classical Laguerre or Meixner weight measure. In the classical
theory of orthogonal polynomials, such functionals are called the moment
functionals. For this reason,  we will call $\varphi^\LA$ and $\varphi^\ME$ the
(formal) {\it moment functionals\/} associated with the Laguerre and Meixner
symmetric functions, respectively.

\subsection{Thoma's simplex and Thoma's cone}

\begin{definition}
(i) The {\it Thoma simplex\/} is the subspace $\Om$ of the infinite product
space $\R_+^\infty\times\R_+^\infty$ formed by all couples $\om=(\al,\be)$,
where $\al=(\al_i)$ and $\be=(\be_i)$ are two infinite sequences such that
\begin{equation}\label{eq20}
\al_1\ge\al_2\ge\dots\ge0, \qquad \be_1\ge\be_2\ge\dots\ge0
\end{equation}
and
\begin{equation}\label{eq21}
\sum_{i=1}^\infty\al_i+\sum_{i=1}^\infty\be_i\le 1.
\end{equation}
We equip $\Om$ with the product topology induced from
$\R_+^\infty\times\R_+^\infty$. Note that in this topology, $\Om$ is a compact
metrizable space.

(ii) The {\it Thoma cone\/} $\wt\Om$ is the subspace of the infinite product
space $\R_+^\infty\times\R_+^\infty\times\R_+$ formed by all triples
$\wt\om=(\al,\be,r)$, where $\al=(\al_i)$ and $\be=(\be_i)$ are two infinite
sequences and $r$ is a nonnegative real number, such that $(\al,\be)$ satisfies
\eqref{eq20} and the following modification of the inequality \eqref{eq21}
$$
\sum_{i=1}^\infty\al_i+\sum_{i=1}^\infty\be_i\le r.
$$
Note that $\wt\Om$ is a locally compact space in the product topology induced
from $\R_+^\infty\times\R_+^\infty\times\R_+$.
\end{definition}

We will identify $\Om$ with the subset of $\wt\Om$ formed by triples
$\wt\om=(\al,\be,r)$ with $r=1$. The name ``Thoma cone'' given to $\wt\Om$ is
justified by the fact that $\wt\Om$ may be viewed as the cone with the base
$\Om$: the ray of the cone passing through a base point $\om=(\al,\be)\in\Om$
consists of the triples $\wt\om=(r\al,r\be,r)$, $r\ge0$.

We are going to realize the algebra $\Sym$ as an algebra of functions on
$\wt\Om$. Since $\Sym$ is freely generated by the elements $p_1,p_2,\dots$, it
suffices to assign to every element $p_k$ a function
$p_k(\wt\om)=p_k(\al,\be,r)$ on $\wt\Om$. This is done as follows:
$$
p_k(\al,\be,r)=\sum_{i=1}^\infty \al_i^k+ (-1)^{k-1}\sum_{i=1}^\infty \be_i^k,
\qquad k=2,3,\dots
$$
(these are the super power sums in variables $(\al_i)$ and $(-\be_i)$, see
\cite[\S I.3, Ex. 23]{Ma95}) and
$$
p_1(\al,\be,r)=r.
$$

\begin{proposition}
The functions $p_1,p_2,\dots$ on\/ $\wt\Om$ defined in this way are continuous
and algebraically independent.
\end{proposition}

\begin{proof}
Because $\sum\al_i\le r$ and $\al_1\ge\al_2\ge\dots\ge0$, we have the bound
$\al_i\le r/i$ and, likewise, $\be_i\le r/i$. Therefore, for any $k\ge2$, the
series $\sum\al_i^k$ and $\sum\be_i^k$ converge uniformly provided that $r$ is
bounded from above. By the definition of the topology in $\wt\Om$, it follows
that the functions $p_k$ are continuous for $k\ge2$.

The continuity of $p_1=r$ is trivial. Note that the special definition of
$p_k(\wt\om)$ in the case $k=1$ can be justified as follows: Consider the
subset
$$
\wt\Om^{\,0}:=\Big\{(\al,\be,r)\in\wt\Om:
\sum_{i=1}^\infty(\al_i+\be_i)=r\Big\}\subset\wt\Om,
$$
which is a dense subset of type $G_\de$. The function $\sum\al_i+\sum\be_i$ is
not continuous on $\wt\Om$, but it coincides with the coordinate function $r$
on $\wt\Om^{\,0}$. Therefore, its unique continuous extension to the whole
space $\wt\Om$ also coincides with $r$.

To prove that the functions $p_1,p_2,\dots$ are algebraically independent,
restrict them on the subset of those $\wt\om=(\al,\be,r)\in\wt\Om$ for which
all beta coordinates are equal to zero, only finitely many of the alpha
coordinates are nonzero, and $r=\sum\al_i$. On this subset, the $p_k$'s turn to
the conventional power sums in $\al_1,\al_2,\dots$, which are well known to be
algebraically independent.
\end{proof}

The proposition provides an embedding of the algebra $\Sym$ into the algebra of
continuous functions on the Thoma cone. Given an element $F\in\Sym$, we use the
notation $F(\wt\om)$ or $F(\al,\be, r)$ for the corresponding function. Vershik
and Kerov \cite{VK90} call such functions {\it extended symmetric functions\/}
in variables $\al=(\al_i)$, $\be=(\be_i)$, and $\ga:=r-\sum(\al_i+\be_i)$. We
prefer to call them {\it polynomial functions on $\wt\Om$\/}. We will identify
elements $F\in\Sym$ with the corresponding polynomial functions $F(\wt\om)$ on
$\wt\Om$. Obviously, if $F\in\Sym$ is homogeneous of degree $n$, then
$F(\wt\om)$ is also homogeneous of degree $n$ as a function on the cone.

\begin{definition}
The {\it Thoma measure\/} associated with a point $\wt\om=(\al,\be,r)\in\wt\Om$
is an atomic measure on $\R$ given by
$$
m_{\wt\om}=\sum \al_i\de_{\al_i}+\sum \be_i\de_{-\be_i}+\gamma\de_0, \qquad
\gamma=r-\sum\al_i-\sum\be_i\ge0,
$$
where $\de_x$ denotes the Dirac mass at a point $x\in\R$.
\end{definition}

Note that $m_{\wt\om}$ has finite mass equal to $r$ and is compactly supported.
Therefore, $m_{\wt\om}$ is uniquely determined by its moments.

\begin{proposition}\label{prop15}
The $k$th moment of the Thoma measure $m_{\wt\om}$ is equal to $p_k(\wt\om)$
for every $k=1,2,\dots$\,.
\end{proposition}

\begin{proof}
Immediate from the definition of the functions $p_k(\wt\om)$.
\end{proof}

Here is a corollary. Note that the restriction of a polynomial function
$F(\wt\om)$ to the subset $\Om\subset\wt\Om$ is a continuous function $F(\om)$
on $\Om$. Since $\Om$ is compact, $F(\om)$ is bounded on $\Om$. Thus, we get an
algebra morphism $\Sym\to C(\Om)$, where $C(\Om)$ is the Banach algebra of
real-valued continuous functions with the supremum norm.

\begin{corollary}\label{cor2}
The image of the map\/ $\Sym\to C(\Om)$ is a dense subalgebra in $C(\Om)$.
\end{corollary}

\begin{proof}
By Proposition \ref{prop15}, the functions $F(\om)$ separate points of the
compact space $\Om$. Then apply  the Stone-Weierstrass theorem.
\end{proof}

The algebra morphism $\Sym\to C(\Om)$ is defined on the generators
$p_1,p_2,\dots$ by the same formulas as above, only $p_1$ is mapped to the
constant function $1$ on $\Om$. It follows that the kernel of the morphism is
the principal ideal in $\Sym$ generated by $p_1-1$.

Let $P$ be a probability measure on $\Om$ (here and below all measures are
tacitly assumed to be Borel measures). Expectation under $P$ determines a
linear functional on $\Sym$, which we denote by $\psi_P$ and call the {\it
moment functional of $P$}:
$$
\psi_P(F)=\langle F,P\rangle=\int_\Om F(\om)P(d\om), \qquad F\in\Sym.
$$
By Corollary \ref{cor2}, the moment functional determines the initial measure
on $\Om$ uniquely.

\begin{proposition}\label{prop14}
A linear functional $\psi:\Sym\to\R$ is a moment functional of a probability
measure $P$ on $\Om$ if and only if $\psi$ is nonnegative on all Schur
functions, equals\/ $1$ on the constant function $1$, and vanishes on the
principal ideal in\/ $\Sym$ generated by $p_1-1$.
\end{proposition}

\begin{proof}
This result is essentially due to Vershik and Kerov \cite{VK81}. It is
equivalent to the special case of Theorem B in \cite{KOO98} corresponding to
special value $\theta=1$ of the Jack parameter (then the Jack symmetric
symmetric functions turn into the Schur symmetric functions). The statement of
that theorem speaks about ``harmonic functions'' \footnote{This terminology,
introduced by Vershik and Kerov in \cite{VK90}, is unfortunate, since it
disagrees with the conventional terminology adopted in potential theory.} on
$\Y$ instead of moment functionals on $\Sym$, but this is an equivalent
language. Indeed, using the formula
$$
p_1 S_\nu=\sum_{\varkappa\in\Y:\, \varkappa=\nu\cup\Box}S_\varkappa
$$
one sees that the function $\varphi_P(\nu):=\psi_P(S_\nu)$ is ``harmonic''  if
and only if $\psi_P$ vanishes on the principal ideal generated by $p_1-1$.
\end{proof}

The assertion of Proposition \ref{prop14} may be viewed as a generalization of
classical Hausdorff's theorem characterizing probability measures on the closed
unit interval $[0,1]$ in terms of their moment functionals on $\R[x]$. Indeed,
embed the closed interval $[0,1]$ into $\Om$ by making use of the map
$$
[0,1]\ni x\mapsto(\al,\be)=((x,0,0,\dots),(1-x,0,0,\dots))\in\Om.
$$
In the particular case when $P$ is concentrated on the image of $[0,1]$,
Proposition \ref{prop14} just reduces to Hausdorff's theorem.

\begin{corollary}\label{cor3}
Every Schur function $S_\nu$ is nonnegative on $\wt\Om$.
\end{corollary}

\begin{proof}
Let $P$ be the delta measure at a point $\om\in\Om$. By Proposition
\ref{prop14}, $\psi_P(S_\nu)=S_\nu(\om)\ge0$. Thus, the function $S_\nu$ is
nonnegative on $\Om\subset\wt\Om$. Since it is a homogeneous function of degree
$|\nu|$, it is nonnegative on the whole cone $\wt\Om$.
\end{proof}

Let $\wt\om_0$ denote the vertex of the Thoma cone: all the coordinates of
$\wt\om_0$ are $0$. There is a natural bijection between the space
$\wt\Om\setminus\{\wt\om_0\}$ and the product space $\Om\times\R_{>0}$:
\begin{equation}\label{eq22}
\wt\om=(\al,\be,r)\quad \leftrightarrow \quad (\om,r), \qquad
\om:=(r^{-1}\al,r^{-1}\be).
\end{equation}

\begin{definition}
Given a probability measure $P$ on the Thoma simplex $\Om$ and a parameter
$c>0$, we construct a probability measure $\wt P^c$ on the cone $\wt\Om$,
depending on $c>0$, as the push--forward of the product measure $P\otimes\ga_c$
under the correspondence $(\om,r)\mapsto\wt\om$ defined in \eqref{eq22} (recall
that $\ga_c$ denotes the gamma distribution with parameter $c$, see
\eqref{eq29}). We call $\wt P^c$ the {\it lifting\/} of the measure $P$ with
the {\it lifting parameter\/} $c$. The initial measure $P$ can be reconstructed
from its lifting $\wt P^c$ by taking the push--forward of $\wt P^c$ under the
projection $\wt\Om\setminus\{\wt\om_0\}\to\Om$.
\end{definition}

\begin{definition}
A function $F(\wt\om)$ on $\wt\Om$ will be called {\it polynomially
bounded\/} if it admits a global bound of the form
$$
|F(\al,\be,r)|\le\const_1\cdot (1+r)^{\const_2}
$$
with appropriate constants.
\end{definition}

\begin{proposition}\label{prop18}
Let $P$ be an arbitrary probability measure on $\Om$ and $\wt P$ be its lifting
to $\wt\Om$ with some parameter $c>0$.

{\rm(i)} For any $F\in\Sym$, the corresponding function $F(\wt\om)$ is
polynomially bounded.

{\rm(ii)} Continuous polynomially bounded functions are integrable with respect
to any measure of the form $\wt P^c$.
\end{proposition}

\begin{proof} (i) It suffices to check the claim for $f=p_k$, and then it is
obvious because $|p_k(\al,\be,r)|\le r^k$ by the very definition of
$p_k(\wt\om)$.

(ii) This immediately follows from the fact that polynomials in $r$ of
arbitrary degree are integrable with respect to the gamma distribution.
\end{proof}

Since $\Sym$ is an algebra, the proposition implies that the functions from
$\Sym$ are not only integrable but are also square integrable. Thus, they are
contained in the Hilbert space $L^2(\wt\Om,\wt P^c)$.

\begin{proposition}\label{prop17}
Let $\wt P^c$ be as above. The functions from\/ $\Sym$ form a dense subspace in
$L^2(\wt\Om,\wt P^c)$.
\end{proposition}

\begin{proof}
Recall that under the bijection \eqref{eq22} between
$\wt\Om\setminus\{\wt\om_0)\}$ and $\Om\times\R_{>0}$, the measure $\wt P$
turns into $P\otimes\ga_c$. Thus, there is a natural isometry of Hilbert spaces
$$
L^2(\wt\Om,\wt P^c)\to L^2(\Om,P)\otimes L^2(\R_+,\ga_c).
$$

Let $\Sym^\circ=\Sym/(p_1-1)$ be the quotient of the algebra $\Sym$ by the
principal ideal generated by $p_1-1$. Given $F\in\Sym$ we denote its image in
$\Sym^\circ$ by $F^\circ$. Clearly, $p^\circ_1=1$ and $\La^\circ$ is freely
generated (as a unital commutative algebra) by $p^\circ_2,p^\circ_3, \dots$.

The algebra morphism $\Sym\to C(\Om)$ defined just before Corollary \ref{cor2}
factors through the quotient algebra $\Sym^\circ$ and determines an embedding
$\Sym^\circ\to C(\Om)$. Under this embedding, the generators
$p_2^\circ,p_3^\circ,\dots$ are converted to the functions
$$
p_k^\circ(\om)=\sum_i\al_i^k+(-1)^{k-1}\sum_j\be_j^k, \qquad
\om=(\al,\be)\in\Om, \quad k=2,3,\dots
$$

By Corollary \ref{cor2}, the functions on $\Om$ coming from $\Sym^\circ$ are
dense in $C(\Om)$. Therefore, they are also dense in $L^2(\Om,P)$.

Thus, it suffices to prove that each function on $\Om\times\R_{>0}$ of the form
$$
f^\circ\otimes g, \qquad \text{where $f^\circ\in\Sym^\circ, \quad  g\in
L^2(\R_{>0},\ga_c)$},
$$
can be approximated by functions coming from $\Sym$, in the norm of the Hilbert
space $L^2(\Om,P)\otimes L^2(\R_+,\ga_c)$.

Without loss of generality we may assume that $f^\circ$ comes from a
homogeneous element $f\in\Sym$. Indeed, take first an arbitrary element
$f\in\Sym$ whose image in $\Sym^\circ$ is $f^\circ$. Writing $f$ as a sum of
homogeneous components and multiplying each component by a suitable power of
$p_1$ we get a homogeneous element. On the other hand, this transformation does
not affect $f^\circ$.

Let us show that $f^\circ\otimes g$ can be approximated by functions from
$\Sym$, which have the form $fh$ with $h$ an appropriate polynomial in $p_1$.

Let $m$ be the degree of the homogeneous element $f\in\Sym$. The function on
$\Om\times\R_{>}$ coming from $fh$ has the form
$$
(fh)(\om,r)=f^\circ(\om)r^mh(r), \qquad (\om,r)\in\Om\times\R_{>},
$$
while the value of $f^\circ g$ at the same point is
$$
(f^\circ g)(\om,r)=f^\circ(\om)g(r).
$$
Therefore, we have reduced the problem to the following claim:

Let $m\in\Z_+$ be fixed; then any function $g\in L^2(\R_{>0},\ga_c)$ can be
approximated by functions of the form $r^mh(r)$, where $h$ is a polynomial.

Let us prove this claim. The squared distance in $L^2(\R_{>0},\ga_c)$ between
the functions $g$ and $r^mh(r)$ is equal to
\begin{gather*}
\int_0^\infty|g(r)-r^mh(r)|^2\ga_c(dr)
=\int_0^\infty\left|\frac{g(r)}{r^m}-h(r)\right|^2r^{2m}\ga_c(dr)\\
=\const\int_0^\infty\left|\frac{g(r)}{r^m}-h(r)\right|^2\ga_{c+2m}(dr).
\end{gather*}
Since the function $g(r)$ is square--integrable with respect to $\ga_c$, the
function $g(r)/r^m$ is square--integrable with respect to $\ga_{c+2m}$. Now,
the desired claim follows from the next one:

The space of polynomials is dense in the $L^2$ space with respect to the
gamma--distribution with arbitrary parameter.

Finally, the latter claim is a standard fact and can be checked, e.~g., in the
following way: The characteristic function of the gamma distribution is
analytic near $0$; it follows that the corresponding moment problem is
determinate (see \cite[Prop. 1.6]{Si98}); and this in turn implies the density
claim (see \cite[Prop. 4.15]{Si98}).
\end{proof}

\begin{proposition}
Let $P$ and $\wt P^c$ be as above, $\psi_P:\Sym\to\R$ be the moment functional
for $P$, and $\psi_{\wt P^c}:\Sym\to\R$ be the moment functional for $\wt P^c$
defined by
$$
\psi_{\wt P^c}(F)=\int_{\wt\Om}F(\wt\om)\wt P^c(d\wt\om), \qquad F\in\Sym.
$$

The two functionals are related in the following way: If $F$ is homogeneous of
degree $n$ then
\begin{equation}\label{eq31}
\psi_{\wt P^c}(F)=(c)_n\psi_P(F).
\end{equation}
\end{proposition}

\begin{proof}
Under the bijection \eqref{eq22}, $F$ turns into the function
$$
F(\om,r)=F^\circ(\om) r^n, \qquad \om\in\Om, \quad r>0.
$$
Therefore,
$$
\psi_{\wt P^c}(F)=\int_\Om F^\circ(\om)P(d\om)\cdot\int_{\R_{>0}}r^n\ga_c(dr).
$$
The first integral equals $\psi_P(F)$ while the second integral equals $(c)_n$.
\end{proof}

\subsection{The orthogonality measure for Laguerre symmetric functions}

Let us say that the couple $(z,z')\in\C^2$ is {\it admissible\/} if $z\ne0$,
$z'\ne0$, and  $(z)_\nu(z')_\nu\ge0$ for all $\nu\in\Y$. Obviously, the set of
admissible values is invariant under symmetries $(z,z')\to(z',z)$ and
$(z,z)\to(-z,-z')$, the latter holds because $(-z)_\nu=(-1)^{|\nu|}(z)_{\nu'}$.

It is not difficult to get an explicit description of the admissible range of
the parameters $(z,z')$, see \cite[Proposition 1.2]{BO06c}. One can represent
it as the union of the following three subsets or {\it series\/}:

\begin{itemize}

\item The {\it principal series\/} is $\{(z,z')\colon z'=\bar
z\in\C\setminus\Z\}$.

\item The {\it complementary series\/} is $\cup_{k\in\Z}\{(z,z')\colon
k<z,z'<k+1\}$.

\item The {\it degenerate series\/} comprises the set
$$
\{(z,z')=(N,N+b-1)\colon N=1,2,\dots;\, b>0\}
$$
together with its images under the symmetry group $\Z_2\times\Z_2$.

\end{itemize}

The reason why  the values $z=0$ and $z'=0$ are forbidden is that then
$(z)_\nu(z')_\nu$ vanishes for all $\nu\ne\varnothing$, which is a trivial
case.

If $(z,z')$ is in the principal or complementary series, then $(z)_\nu(z')_\nu$
is strictly positive for all $\nu$. If $(z,z')$ is in the degenerate series,
then $(z)_\nu(z')_\nu$ is strictly positive for diagrams $\nu$ contained in a
horizontal or vertical strip, and vanishes otherwise.

Note that $zz'>0$ for any admissible couple.

\begin{proposition}\label{prop16}
Assume $(z,z)$ is admissible and write $\varphi_{z,z'}^\LA$ for the Laguerre
moment functional with parameters $(z,z')$.

There exists a unique probability measure $P=P_{z,z'}$ on the Thoma simplex\/
$\Om$ such that if $F\in\Sym$ is homogeneous of degree $n$, then
\begin{equation}\label{eq30}
\psi_P(F)=\frac1{(zz')_n}\,\varphi^\LA_{z,z'}(F).
\end{equation}
\end{proposition}

According to the definition of $\varphi^\LA$ this means that
$$
\psi_P(S_\nu) =\frac{(z)_\nu(z')_\nu}{(zz')_{|\nu|}}\,\frac{\dim\nu}{|\nu|!},
\qquad \forall\nu\in\Y.
$$

\begin{proof}
We have to check that the linear functional $\psi_P$ as defined above satisfies
the three conditions stated in Proposition \ref{prop14}. The first condition
(nonnegativity on the Schur functions) holds by the very definition of
admissible parameters. The second condition (normalization) is obvious. The
third condition (vanishing on the principal ideal generated by $p_1-1$ is a
nontrivial claim, for which there exist a few different proofs, see e.g.
\cite[\S8]{Ol03a}, \cite{BO00b}.
\end{proof}

Let us say that a probability measure $\wt P$ on $\wt\Om$ is an {\it
orthogonality measure\/} for the Laguerre symmetric functions with given
parameters $(z,z')$ if the formal moment functional $\varphi^\LA_{z,z'}$ serves
as the moment functional for $\wt P$.

\begin{theorem}\label{thm4}
Let $(z,z')$ be admissible, $P_{z,z'}$ be the probability measure on\/ $\Om$
defined in Proposition \ref{prop16}, and $\wt P_{z,z'}$ be its lifting with
parameter $c=zz'$.

{\rm(i)} The measure $\wt P_{z,z'}$ is an orthogonality measure for the
Laguerre symmetric functions with the same parameters $(z,z')$.

{\rm(ii)} If $(z,z')$ is in the principal or complementary series, then the
Laguerre symmetric functions form an orthogonal basis in the Hilbert space
$L^2(\wt\Om,\wt P_{z,z'})$.
\end{theorem}

\begin{proof}
This follows from \eqref{eq30} and \eqref{eq31}.
\end{proof}

\subsection{The Schur measures}

By a {\it specialization\/} of the algebra $\Sym$ we mean a multiplicative
homomorphism $\psi\colon\Sym\to\C$, $\psi(1)=1$. It is uniquely determined by
its values $\psi(p_k)$ on the generators $p_k$, $k=1,2,\dots$\,. Since the
$p_k$'s are algebraically independent, these values may be chosen arbitrarily
in $\C$. We write $\bar\psi$ for the conjugate specialization defined by
$\bar\psi(p_k)=\overline{\psi(p_k)}$.

We start with a simple technical proposition.

\begin{proposition}\label{prop11}
Assume $\psi$ and $\psi'$ are two specializations of\/ $\Sym$ such that for
some constants $C>0$ and $\eta>0$
\begin{equation}\label{eq15}
|\psi(p_k)|\le C\eta^k, \quad |\psi'(p_k)|\le C\eta^k,\qquad \forall k=1,2,
\dots\,.
\end{equation}
Then for any $n=1,2,\dots$
\begin{equation*}
\sum_{\la\in\Y_n}|\psi(S_\la)\psi'(S_\la)|\le \eta^{2n}\frac{(C^2)_n}{n!}
\end{equation*}
\end{proposition}

\begin{proof}
By Cauchy's inequality,
\begin{equation*}
\sum_{\la\in\Y_n}|\psi(S_\la)\psi'(S_\la)|\le
\frac12\left(\sum_{\la\in\Y_n}\psi(S_\la)\bar\psi(S_\la)
+\sum_{\la\in\Y_n}\psi'(S_\la)\bar\psi'(S_\la)\right).
\end{equation*}
Therefore, with no loss of generality we may assume $\psi'=\bar\psi$.

Denote by $(\,\cdot\,,\,\cdot\,)$ the canonical inner product in $\Sym$. The
Schur functions form an orthonormal basis and the power sum functions $p_\rho$
(where $\rho$ ranges over the set of partitions) form an orthogonal basis. As
well known,
\begin{equation*}
(p_\rho,p_\rho)=\prod_{k=1}^\infty k^{m_k}m_k!,
\end{equation*}
where $m_k$ stands for the multiplicity of $k$ in the partition $\rho$.

Now, we have
\begin{multline*}
\sum_{\la\in\Y_n}\psi(S_\la)\bar\psi(S_\la)
=\sum_{|\rho|=n}\frac{\psi(p_\rho)\bar\psi(p_\rho)}{(p_\rho,p_\rho)}\\ \le
\eta^{2n}\sum_{1\cdot m_1+2\cdot
m_2+\dots=n}\frac{(C^2)^{m_1+m_2+\dots}}{\prod_k
k^{m_k}m_k!}=\eta^{2n}\frac{(C^2)_n}{n!}.
\end{multline*}
Here the last equality follows from the generating series
\begin{gather*}
\sum_{m_1,m_2,\ldots=0}^\infty\frac{(C^2)^{m_1+m_2+\dots}\,t^{1\cdot m_1+2\cdot
m_2+\dots}}{\prod_k k^{m_k}m_k!}
=\prod_{k=1}^\infty\sum_{m_k=0}^\infty\frac{(C^2t/k)^{m_k}}{m_k!}\\
=\prod_{k=1}^\infty\exp(C^2t^k/k)=\exp\left(\sum_{k=1}^\infty\frac{C^2t^k}{k}\right)=(1-t)^{-C^2}.
\end{gather*}
\end{proof}

\begin{corollary}
If the constant $\eta$ in \eqref{eq15} satisfies $\eta<1$, then
$$
\sum_{\la\in\Y}|\psi(S_\la)\psi'(S_\la)|<\infty.
$$
\end{corollary}

\begin{definition}
Let $\psi$ and $\psi'$ be two specializations of the algebra\/ $\Sym$
satisfying condition \eqref{eq15} with a constant $\eta<1$. Set
\begin{equation*}
Z(\psi,\psi')=\sum_{\la\in\Y}\psi(S_\la)\psi'(S_\la)
\end{equation*}
and assume that $Z(\psi,\psi')\ne0$. The {\it Schur measure\/} \cite{Ok01}
corresponding to the couple $(\psi,\psi')$ is the  complex measure
$P^\Schur_{\psi,\psi'}$ on the set $\Y$ with weights
\begin{equation*}
P^\Schur_{\psi,\psi'}(\la)=\frac{\psi(S_\la)\psi'(S_\la)}{Z(\psi,\psi')},
\qquad \la\in\Y.
\end{equation*}
\end{definition}

\smallskip

Observe that if $\psi(S_\la)\psi'(S_\la)\ge0$ for all $\la$, then automatically
$Z(\psi,\psi')\ne0$ and $P^\Schur_{\psi,\psi'}$ is a probability measure.

\subsection{The z-measures}

We will deal with a 3-parameter family of measures on $\Y$, which are a special
case of Schur measures.

\begin{definition}
Let $z$, $z'$, and $\xi$ be complex parameters, $|\xi|<1$. The associated
complex measure on $\Y$, called the (mixed) {\it z-measure\/}, is defined by
\begin{equation}\label{eq28}
P_\zxi(\la)=(1-\xi)^{zz'}(z)_\la(z)_\la\xi^{|\la|}\left(\frac{\dim\la}{|\la|!}\right)^2,
\qquad \la\in\Y.
\end{equation}
\end{definition}

\begin{proposition}
Let $\psi_{z,\eta}$ denote the specialization defined by
$\psi_{z,\eta}(p_k)=z\eta^k$ for all $k=1,2,\dots$\,. The z-measure $P_\zxi$
coincides with the Schur measure corresponding to the couple
$\psi=\psi_{z,\sqrt\xi}$, $\psi'=\psi_{z',\sqrt\xi}$, with arbitrary choice of
the square root.
\end{proposition}

Equivalently, one could take $\psi=\psi_{z,1}$ and $\psi'=\psi_{z',\xi}$.

\begin{proof}
The following identity holds
\begin{equation*}
\psi_{z,1}(S_\la)=(z)_\la\frac{\dim\la}{|\la|!}, \qquad z\in\C.
\end{equation*}
Indeed, in the particular case $z=N=1,2,\dots$ the left-hand side equals
$S_\la(1,\dots,1)$ with $N$ units, and the equality follows from the comparison
of the ``hook formulas'' for $\dim\la$ and $S_\la(1,\dots,1)$, see \cite[\S
I.3, Ex. 4]{Ma95}. Then the general case follows from the observation that the
both sides are polynomials in $z$.

The identity implies that
\begin{equation*}
\psi_{z,\eta}(S_\la)\psi_{z',\eta}(S_\la)
=(z)_\la(z')_\la(\eta^2)^{|\la|}\left(\frac{\dim\la}{|\la|!}\right)^2.
\end{equation*}
It remains to show that if $|\eta|<1$, then
\begin{equation*}
Z(\psi_{z,\eta},\psi_{z',\eta})=(1-\eta^2)^{-zz'}.
\end{equation*}
This follows from the identity
\begin{equation*}
\sum_{\la\in\Y_n}\psi_{z,1}(S_\la)\psi_{z',1}(S_\la)=\frac{(zz')_n}{n!},
\end{equation*}
which is verified by exactly the same computation as in the proof of
Proposition \ref{prop11}.
\end{proof}

\begin{remark}\label{rem2}
(i) Assume $zz'\ne0,-1,-2,\dots$ and consider, for arbitrary $n=0,1,2,\dots$,
the complex measure on the finite set $\Y_n$ defined by
$$
P^{(n)}_{z,z'}(\la)=\frac{(z)_\la(z')_\la}{(zz')_n}\,\frac{(\dim\la)^2}{n!},
\quad \la\in\Y_n.
$$
The above computation shows that
$$
\sum_{\la\in\Y_n}P^{(n)}_{z,z'}(\la)=1.
$$
In particular, if the quantities $(z)_\la(z')_\la$ are nonnegative for all
$\la$, which happens for admissible $(z,z')$, then $P^{(n)}_{z,z'}$ is a
probability measure on $\Y_n$.

(ii) Obviously,
$$
P_\zxi(\la)=(1-\xi)^{zz'}\frac{(zz')_{|\la|}}{|\la|!}\xi^{|\la|}\cdot
P^{(|\la|)}_{z,z'}(\la).
$$
This means that the (mixed) z-measure $P_\zxi$ is obtained by mixing out the
measures $P^{(n)}_{z,z'}$ for different $n$ by means of the measure on $\Z_+$
with the weights
$$
n\to(1-\xi)^{zz'}\frac{(zz')_n}{n!}\xi^n, \qquad n=0,1,2,\dots\,.
$$
The later measure is a negative binomial distribution provided that $zz'>0$ and
$\xi\in(0,1)$. It follows that if $(z,z')$ is admissible and $\xi\in(0,1)$,
then $P_\zxi$ is a probability measure on $\Y$.

(iii) The measures $P^{(n)}_{z,z'}$ are called the (non-mixed) {\it
z-measures\/}. They first appeared in \cite{KOV93}; see also \cite{KOV04}. The
mixed z-measures $P_\zxi$ were introduced in \cite{BO00a} and probably served
as a guiding example for the introduction of the general Schur measures in
\cite{Ok01}. In a different context, examples of Schur--type measures related
to Macdonald polynomials appeared earlier in \cite{Fu97}. For additional
information about the z-measures, see \cite{Ol03a}, \cite{BO00b}, \cite{BO06a},
\cite{BO06c}, \cite{BO09}, \cite[Example 3]{BOk00}.
\end{remark}

\begin{proposition}
Let $F(\la)$ be a function on $\Y$ satisfying a bound of the form
\begin{equation}\label{eq23}
|F(\la)|\le\const(1+|\la|)^m, \qquad \forall\la\in\Y
\end{equation}
with some $m>0$. Then the series
\begin{equation*}
\sum_{\la\in\Y}F(\la)P_\zxi(\la)
\end{equation*}
converges absolutely and uniformly with respect to parameters $z,z',\xi$
provided that $z$ and $z'$ range in a bounded region of\/ $\C$ and $\xi$ ranges
in a disc $|\xi|\le 1-\epsi$ with $\epsi>0$.
\end{proposition}

\begin{proof}
The claim reduces to analysis of convergence of the double series
\begin{equation*}
\sum_{n=1}^\infty
n^m|\xi|^n\sum_{\la\in\Y_n}|\psi_{z,1}(S_\la)\psi_{z',1}(S_\la)|.
\end{equation*}
The computation in Proposition \ref{prop11} shows that the interior sum is
bounded by $(C^2)_n/n!$, where $C=\max{|z|,|z'|}$, and the claim becomes
evident.
\end{proof}

Using the isomorphism $\Sym\to\A$ (see \eqref{eq32}) we will regard elements of
$\Sym$ as functions on $\Y$. With this agreement we have

\begin{proposition}[cf. Proposition \ref{prop17}]\label{prop12}
Assume that $(z,z')$ is admissible and $0<\xi<1$, so that $P_\zxi$ is a
probability measure on $\Y$. Then all functions from\/ $\Sym$ are
square-integrable with respect to $P_\zxi$. Moreover, these functions form a
dense subspace in the real Hilbert space $\ell^2(\Y,P_\zxi)$.
\end{proposition}

\begin{proof}
Denote by $\BIN_{c,\xi}$ the negative binomial distribution on $\Z_+$ with
parameters $c>0$ and $\xi\in(0,1)$:
\begin{equation*}
\BIN_{c,\xi}(n)=(1-\xi)^c\frac{(c)_n}{n!}\xi^n, \qquad n\in\Z_+.
\end{equation*}
We have already encountered with this distribution as the weight measure for
the classical Meixner polynomials, but in the present context, it plays a
somewhat different r\^ole, and we will use a different notation.

Next, let us abbreviate $P=P_\zxi$. As is seen from Remark \ref{rem2} (ii), the
push--forward of $P$ under the projection $\Y\to\Z_+$ sending $\la$ to $|\la|$
is $\BIN_{c,\xi}$, where $c=zz'>0$. This is the only property of $P$ that we
actually need, so that in the argument below $P$ may be an arbitrary
probability measure on $\Y$ with such a property.

The first two steps below are parallel to Proposition \ref{prop18}.

{\it Step 1}. Every function $F\in\Sym$ satisfies a bound of the form
\eqref{eq23}. Indeed, it suffices to check this for the generators $p_k$, and
then \eqref{pol.B} gives
\begin{equation*}
|p_k(\la)|\le\big(\sum_i(a_i+b_i)\big)^k=|\la|^k.
\end{equation*}

{\it Step 2}. Since $\Sym$ is an algebra, the first claim is equivalent to
saying that the functions from $\Sym$ are $P$-integrable. By virtue of step 1
we may assume that $|F(\la)|$ grows not faster as $|\la|^m$. Then the claim
about integrability reduces to the obvious fact that the function $n\to n^m$ on
$\Z_+$ is summable with respect to the negative binomial distribution.

{\it Step 3}. To prove that $\Sym$ is dense in $\ell^2(\Y,P)$ it suffices to
verify that any given function on $\Y$ with finite support can be approximated,
in the metric of $\ell^2(\Y,P)$, by elements from $\Sym$. Take $l$ so large
that our function is supported by the set $\Y_{\le l}=\{\la\in\Y\mid |\la|\le
l\}$. Let $\chi_l$ be the characteristic function of $\Y_{\le l}$. Using the
interpolation property of Frobenius--Schur functions, we can find an element
$f\in\Sym$ taking any prescribed values on $\Y_{\le l}$. Therefore, it suffices
to approximate any function of the form $f\chi_l$ with given $l$ and
$f\in\Sym$.

{\it Step 4}.  We will show that $f\chi_l$ can be approximated by elements of
the form $f\cdot h(p_1)\in\Sym$, where $h$ are appropriate polynomials in one
indeterminate. The idea is to reduce this to a one-dimensional problem.

Denoting by $m$ the degree of $f$ we have a bound
$$
|f(\la)|^2\le \const(1+|\la|^{2m})
$$ with an appropriate constant in front. Next, let
$\Vert\,\cdot\,\Vert$ denote the norm in $\ell^2(\Y,P)$, and let $\bar\chi_l$
be the characteristic function of the set $\{0,1,\dots,l\}\subset\Z_+$. We have
\begin{multline}\label{jump.B}
\Vert f\chi_l-f\cdot h(\,\cdot\,)\Vert^2=\Vert
f(\chi_l-h(\,\cdot\,))\Vert^2 \\
\le\const\sum_{n\in\Z_+}(1+n^{2m})|\bar\chi_l(n)-h(n)|^2\BIN_{c,\xi}(n).
\end{multline}

We claim that there exists a bound of the form
$$
(1+n^{2m})\BIN_{c,\xi}(n)\le\const\BIN_{c+2m,\xi}(n), \qquad n\in\Z_+,
$$
with another appropriate constant in front. Indeed, this follows from two
facts: first, $\BIN_{c,\xi}(n)$ is strictly positive for all $n\in\Z_+$;
second, for large $n$
$$
\frac{\BIN_{c+2m,\xi}(n)}{\BIN_{c,\xi}(n)}
=\const\,\frac{\Ga(c+2m+n)}{\Ga(c+n)}\sim\const n^{2m}.
$$

Therefore, \eqref{jump.B} is bounded from above by
$$
\const\sum_{n\in\Z_+}|\bar\chi_l(n)-h(n)|^2\BIN_{c+2m,\xi}(n)
$$
with a new constant factor. Now, choosing appropriate $h$ we can make this
expression arbitrarily small, because the polynomials are dense in the space
$\ell^2(\Z_+,\BIN_{c+2m})$. The last claim is well known: it can be derived,
e.g., from the fact that the Laplace transform of $\BIN_{c+2m}$ is well defined
in a neighborhood of the origin, cf. the end of proof of Proposition
\ref{prop17}.
\end{proof}

\subsection{The orthogonality measure for Meixner symmetric functions}

As above, we interpret elements of $\Sym$ as functions on $\Y$. Recall that  in
Definition \ref{def2} we have introduced the formal moment functional
$\varphi^\ME$ associated with the Meixner symmetric functions $\MM_\nu$. Now it
is convenient to use for it a more detailed notation $\varphi^\ME_{z,z',\xi}$.

\begin{theorem}[cf. Theorem \ref{thm4}]\label{thm3}
Let $(z,z')$ be admissible and $\xi\in(0,1)$.

{\rm(i)} The measure $P_\zxi$ is an orthogonality measure for the Meixner
symmetric functions with the same parameters $\zxi$, meaning that the formal
moment functional $\varphi^\ME_{z,z',\xi}$ on\/ $\Sym$ coincides with
expectation under the z-measure $P_\zxi${\rm:}
\begin{equation}\label{eq16}
\varphi^\ME_\zxi(F)=\sum_{\la\in\Y}F(\la)P_\zxi(\la), \qquad \forall F\in\Sym.
\end{equation}

{\rm(ii)} If $(z,z')$ is in the principal or complementary series, then the
Meixner symmetric functions form an orthogonal basis in the real weighted
Hilbert space $\ell^2(\Y,P_\zxi)$.
\end{theorem}

\begin{proof}
(i) It suffices to check \eqref{eq16} for $F=FS_\nu$. Fix $\nu\in\Y_n$ and
compute the right-hand side.

The measures $P^{(n)}_{z,z'}$  satisfy an important relation, called the {\it
coherency relation\/}, see e.g. \cite[Proposition 1.2]{BO06a}:
$$
P^{(n)}_{z,z'}(\nu)=\sum_{\la\in\Y_{n+1}:\,
\la\supset\nu}\frac{\dim\nu}{\dim\la}\, P^{(n+1)}_{z,z'}(\la), \qquad
\nu\in\Y_n.
$$
Iterating this relation we get for any $l\ge n$
$$
P^{(n)}_{z,z'}(\nu)=\sum_{\la\in\Y_l:\,
\la\supseteq\nu}\frac{\dim\nu\dim(\la/\nu)}{\dim\la}\, P^{(l)}_{z,z'}(\la),
\qquad \nu\in\Y_n.
$$
Multiplying the both sides by
\begin{equation*}
\frac1{\dim\nu}(1-\xi)^{zz'}\frac{(zz')_l}{l!}\xi^l
\end{equation*}
gives
\begin{equation}\label{eq33}
\frac{(zz'+n)_{l-n}}{l!}\xi^{l-n}\cdot\frac{n!\,P_{z,z',\xi}(\nu)}{\dim\nu}
=\sum_{\la\in\Y_l:\, \la\supseteq\nu}\frac{\dim(\la/\nu)}{\dim\la}\,
P_{z,z',\xi}(\la).
\end{equation}

Recall (see \eqref{eq17} above) that for any $\nu\in\Y_n$ and $\la\in\Y_l$
$$
FS_\nu(\la)=\begin{cases}\dfrac{l!}{(l-n)!}\dfrac{\dim(\la/\nu)}{\dim\la}, &
\nu\subseteq\la, \\
0, &\textrm{otherwise}\end{cases}
$$
In particular, $FS_\nu(\la)$ vanishes unless $l\ge n$ and $\la\supseteq\nu$.
Combining this with \eqref{eq33} we get
$$
\sum_{\la\in\Y}FS_\nu(\la)P_\zxi(\la)
=\frac{n!\,P_{z,z',\xi}(\nu)}{\dim\nu}\,\sum_{l\ge
n}\frac{(zz'+n)_{l-n}}{(l-n)!}\xi^{l-n}.
$$
But the latter sum equals $(1-\xi)^{-zz'-n}$, so that
$$
\begin{aligned}
\sum_{\la\in\Y}FS_\nu(\la)P_\zxi(\la)
&=\frac{n!\,P_{z,z',\xi}(\nu)}{\dim\nu}\,(1-\xi)^{-zz'-n}\\
&=\left(\frac\xi{1-\xi}\right)^n (z)_\nu(z')_\nu\frac{\dim\nu}{n!}.
\end{aligned}
$$
This agrees with the expression for the left-hand side of \eqref{eq16} given in
\eqref{eq13}.

Note that the above argument holds under weaker assumptions on the parameters:
it suffices to assume that $(z,z',\xi)\in\C^3$ and $|\xi|<1$.

(ii) Proposition \ref{prop12} says that the functions from $\Sym$ form a dense
subspace in the Hilbert space $\ell^2(\Y,P_\zxi)$. By virtue of (i) and Theorem
\ref{thm2}, the functions $\MM_\nu$ are pairwise orthogonal. Finally, since
$(z,z')$ is not in the degenerate series, the expression \eqref{eq11} for the
squared norm $(\MM_\nu,\MM_\nu)$ is strictly positive for all $\nu$. Therefore,
the functions $\MM_\nu$ form an orthogonal basis in $\ell^2(\Y,P_\zxi)$.
\end{proof}

Note that in the case when $(z,z')$ is in the degenerate series, the squared
norm $(\MM_\nu,\MM_\nu)$ vanishes for some $\nu$. This  means that the
corresponding functions $\MM_\nu$ vanish on the support of the measure
$P_\zxi$, so that these functions produce zero vectors in the Hilbert space
$\ell^2(\Y,P_\zxi)$. But the remaining elements $\MM_\nu$ still form an
orthogonal basis.

\subsection{Limit transition Meixner $\to$ Laguerre for orthogonality measures}

Using Frobenius' coordinates of Young diagrams we define a family of embeddings
$\iota_\epsi\colon\Y\to\wt\Om$ depending on the parameter $\epsi>0$:
$$
\Y\ni \la=(a;b)\quad\mapsto\quad\wt\om=(\al,\be,r)\in\wt\Om,
$$
where
\begin{equation}
\begin{gathered}
\al=(\epsi a_1, \dots, \epsi a_d, 0,0,\dots),\\
\be=(\epsi b_1, \dots, \epsi b_d, 0,0,\dots),\\
r=\epsi|\la|.
\end{gathered}
\end{equation}

The image $\Y^{(\epsi)}=\iota_\epsi(\Y)$ is contained in $\wt\Om^0$ and is a
discrete subset of $\wt\Om$ which becomes more and more dense as $\epsi\to0$.
We regard $\Y^{(\epsi)}$ with small $\epsi$ as a grid approximation of the
Thoma cone, just as the lattice $\epsi\Z$ forms a grid approximation of the
real line $\R$.

Recall that the Thoma simplex $\Om$ was defined as the compact subset of
$\wt\Om$ consisted of triples $(\al,\be,r)$ with $r=1$. For every
$n=1,2,\dots$, the map $\iota_{1/n}$ determines an embedding $\Y_n\to\Om$. As
$n\to\infty$, the finite sets $\iota_{1/n}(\Y_n)$ become more and more dense in
$\Om$; we regard them as a grid approximation of the Thoma simplex.

\begin{theorem}\label{thm5}
Let $(z,z')$ be admissible and $\xi=1-\epsi\in(0,1)$. As $\xi\to1$, that is,
$\epsi\to0$, the push--forwards $\iota_\epsi(P_{z,z',1-\epsi})$ of the mixed
z--measures converge to measure $\wt P_{z,z'}$ on arbitrary continuous,
polynomially bounded test functions on $\wt\Om$.
\end{theorem}

This implies, in particular, that $\iota_\epsi(P_{z,z',1-\epsi})\to\wt
P_{z,z'}$ weakly.

\begin{proof}
As it will be shown, the result holds in a more general context. Start with an
arbitrary probability measure $P$ on the Thoma simplex $\Om$. Fix an arbitrary
$c>0$ and consider the lifting of $P$ with parameter $c>0$; this is a
probability measure on $\wt\Om$, which we will denote by $\wt P$.

On the other hand, $P$ gives rise to a sequence $\{P^{(n)}: n=0,1,2,\dots\}$ of
probability measures, where $P^{(n)}$ lives on the finite set $\Y_n\subset\Y$
and is defined by
$$
P^{(n)}(\la)=\dim\la\,\psi_P(S_\la), \qquad \la\in\Y_n.
$$
The fact that this is indeed a probability measure follows from Proposition
\ref{prop14}. Next, we mix up the measures $P^{(n)}$ by means of the negative
binomial distribution $\BIN_{c,1-\epsi}$. The result is a probability measure
on $\Y$, we take its push-forward under the embedding $\iota_\epsi$, and denote
the resulting probability measure on $\wt\Om$ by $\wt P_\epsi$.

We are going to prove that  $\wt P_\epsi$ converges to $\wt P$ on continuous
polynomially bounded test functions. In the special case $P=P_{z,z'}$ this is
exactly the claim of the theorem, because then $\wt P=\wt P_{z,z'}$, $P^{(n)}$
coincides with $P^{(n)}_{z,z'}$, and $\wt P_\epsi$ coincides with
$\iota_\epsi(P_{z,z',1-\epsi})$.

{\it Step 1\/}. Consider the measure $P^{(n)}$ and take its push-forward
$\iota_{1/n}(P^{(n)})$, which is a measure on the Thoma simplex $\Om$. As
$n\to\infty$, $\iota_{1/n}(P^{(n)}$ weakly converges to $P$. Indeed, this is a
particular case of a  more general result: see \cite{KOO98}.

{\it Step 2\/}. Denote by $\wt{\BIN}_{c,1-\epsi}$ the push-forward of the
negative binomial distribution under the embedding $\Z_+\to\R_+$ taking
$n\in\Z_+$ to $\epsi n\in\R_+$. Let $\epsi$ goes to 0. We claim that then
$\wt{\BIN}_{c,1-\epsi}$ converges to the gamma distribution
$$
\ga_c(dr)=\frac1{\Ga(c)}r^{c-1}e^{-r}dr, \qquad r\in\R_+\,,
$$
on continuous test functions on $\R_+$ with at most polynomial growth at
$+\infty$.

Indeed, for $k=1,2,\dots$, the $k$th factorial moment of $\BIN_{c,1-\epsi}$ is
given by the expression
$$
\int_0^{+\infty}x(x-1)\dots(x-k+1)\wt{\BIN}_{c,1-\epsi}(dx)
=(c)_k\left(\frac\xi{1-\xi}\right)^k\,,
$$
which can be rewritten as
$$
\int_0^{+\infty}r(r-\epsi)\dots(r-(k-1)\epsi)\wt{\BIN}_{c,1-\epsi}(dr)
=(c)_k\xi^k\,.
$$
Therefore, all the moments of $\wt{\BIN}_{c,1-\epsi}$ exist and converge, as
$\epsi\to0$, to the respective moments of $\ga_c$. Note also that
characteristic function of the gamma distribution is analytic near 0, so that
$\ga_c$ is a unique solution to the corresponding moment problem. This readily
implies the desired claim.

{\it Step 3\/}. We proceed to the proof that for any continuous, polynomially
bounded function $F$ on $\wt\Om$
\begin{equation}\label{eq24}
\lim_{\epsi\to0}\int_{\wt\Om}F(\wt\om)\wt P_\epsi(d\wt\om)
=\int_{\wt\Om}F(\wt\om)\wt P(d\wt\om).
\end{equation}
By virtue of Step 2, it suffices to prove \eqref{eq24} in the case when $F$ is
compactly supported. Indeed, consider the projection $\wt\Om\to\R_+$ taking
$\wt\om=(\al,\be,r)$ to $r$. Under this projection, the push-forwards of the
measures $\wt P_\epsi$ and $\wt P$ are the $\wt{\BIN}_{c,1-\epsi}$ and $\ga_c$,
respectively. Then the result of Step 2 allows us to  neglect the tails of the
measures.

Thus, we may assume that $F$ vanishes for $r$ large enough.

{\it Step 4\/}. Again by virtue of Step 2, \eqref{eq24} holds true when $F$
depends on $r$ only. Subtracting from $F$ an appropriate function depending on
$r$ we may assume that $F$ vanishes at the point $\wt\om_0$, the vertex of the
Thoma cone.

{\it Step 5\/}. Let us identify $\wt\Om\setminus\{\wt\om_0\}$ with the product
space $\Om\times\R_{>0}$ by means of the correspondence
$$
\wt\om=(\al,\be,\de)\,\leftrightarrow\, (\om,r), \qquad r:=\de>0,\quad
\om=\de^{-1}\wt\om\in\Om.
$$

As $F$ is compactly supported and $F(\wt\om_0)=0$, it can be approximated, in
the supremum norm, by linear combinations of the factorized functions of the
form $F(\om,r)=f(\om)g(r)$, where $f$ is a continuous function on the compact
space $\Om$ and $g(r)$ is a continuous function on $\R_{>0}$ vanishing outside
a closed interval $[a,b]$ with $0<a<b$.

{\it Step 6\/}. After these simplifications, the limit relation \eqref{eq24} is
easily derived from the claims  of Steps 1 and 2.

Indeed, given  $r$ contained in the support $\epsi\Z_+$ of the distribution
$\wt{\BIN}_{c,1-\epsi}$, we define $n(r)\in\Z_+$ from the relation $r=\epsi
n(r)$. If $r$ is bounded from below, which is our case, then $n(r)$ goes to
infinity, uniformly on $r$, as $\epsi\to0$. Then we have
$$
\int_{\wt\Om}F(\wt\om)\wt P_\epsi(d\wt\om) =\int_a^b \wt{\BIN}_{c,1-\epsi}(dr)
\int_{\Om}f(\om)\iota_{1/n(r)}(P^{(n(r))})(d\om)
$$
As $\epsi$ gets small, the interior integral in the right--hand side becomes
close to
$$
\int_{\Om}f(\om)P(d\om),
$$
uniformly on $r\in[a,b]$ (this follows from Step 1). Together with the result
of Step 2 this implies that the right--hand side converges to a product of two
integrals:
$$
\int_a^bg(r)\ga_{c}(dr)\cdot \int_{\Om}f(\om)P(d\om),
$$
which coincides with the right--hand side of the desired limit relation, by the
very definition of lifting.
\end{proof}

\section{Appendix: The Charlier symmetric functions}

Let $\eta_\tth$ denote the Poisson distribution on $\Z_+$ with parameter
$\tth>0$:
$$
\eta_\tth=e^{-\tth}\sum_{x\in\Z_+}\frac{\tth^x}{x!}\,\de_x.
$$
The {\it classical Charlier polynomials\/} with parameter $\tth$ are the
orthogonal polynomials with the weight measure $\eta_\tth$. The monic Charlier
polynomials are given by
$$
C_n(x)=\sum_{m=0}^n (-\tth)^{n-m}\,\frac{n^{\down m}}{m!}\, x^{\down m}.
$$
These are just the polynomial eigenfunctions of the difference operator
$$
D^\Ch f(x)=\tth f(x+1)+xf(x-1)-(\tth+x)f(x).
$$

The polynomials $C_n(x)$ can be obtained as the degeneration of the Meixner
polynomials $M_n(x)$ in the following limit regime for the Mexiner parameters
$(b,\xi)$
$$
b\to+\infty, \quad\xi\to0, \quad b\xi\to\tth.
$$

All the definitions and results concerning the symmetric Meixner polynomials
and the symmetric Meixner functions extend, with simplifications, to the
Charlier case. Below we list the main formulas; they are obtained by
degeneration from the corresponding formulas for the Meixner functions as
$$
z\to\infty, \quad z'\to\infty, \quad \xi\to0, \quad zz'\xi\to\tth.
$$

The {\it Charlier symmetric functions\/} with parameter $\tth$ are given by the
following expansion in the Frobenius--Schur symmetric functions, cf.
\eqref{eq6}:
\begin{equation*}
\mathfrak C_\nu=\sum_{\mu:\,\mu\subseteq\nu} (-\tth)^{|\nu|-|\mu|}\,
\frac{\dim\nu/\mu}{(|\nu|-|\mu|)!}\cdot \FS_\mu.
\end{equation*}
The {\it Charlier operator\/} in $\Sym$ is a degeneration of the Meixner
operator \eqref{eq8}:
\begin{equation}\label{eq25}
\D^\Ch \FS_\nu=-|\nu|\FS_\nu+\tth
\sum_{\square\in\nu^-}\FS_{\nu\setminus\square}
\end{equation}
We have
$$
\D^\Ch\, \mathfrak C_\nu=-|\nu|\mathfrak C_\nu.
$$
The orthogonality measure for the Charlier symmetric functions is the {\it
poissonized Plancherel measure\/} on $\Y$ (cf. \eqref{eq28}):
$$
P_\tth(\la)=e^{-\tth}\tth^{|\la|}\left(\frac{\dim \la}{|\la|}\right)^2, \qquad
\la\in\Y.
$$
The Charlier functions form an orthogonal basis in the weighted Hilbert space
$\ell^2(\Y,P_\tth)$. The squared norm of $\mathfrak C_\nu$ is given by formula
(cf. \eqref{eq11})
$$
(\mathfrak C_\nu,\mathfrak C_\nu)=\tth^{|\nu|}.
$$
The {\it moment functional\/} corresponding to $P_\tth$ has the form (cf.
\eqref{eq13})
$$
\varphi^\Ch(\FS_\nu)=\tth^{|\nu|}\,\frac{\dim\nu}{|\nu|!}.
$$

\end{document}